\documentclass{amsart}

\usepackage{amssymb}
\usepackage{amsmath}
\usepackage{amsthm}
\usepackage{amsfonts}

\usepackage{a4wide}
\usepackage{longtable}
\usepackage{multirow}

\newtheorem{proposition}{Proposition}[section]
\newtheorem{theorem}{Theorem}[section]
\newtheorem{corollary}{Corollary}[section]

\theoremstyle{definition}
\newtheorem{definition}{Definition}[section]

\theoremstyle{remark}
\newtheorem{remark}{Remark}[section]
\newtheorem{example}{Example}[section]

\begin{document}

\title{\(3\)-class field towers with \(2\) or \(3\) stages}

\author{Helga Boyer von Berghof}
\address{Krenngasse 43\\8010 Graz\\Austria}
\email{helgaboyervonberghof@gmail.com}

\author{Daniel C. Mayer}
\address{Naglergasse 53\\8010 Graz\\Austria}
\email{quantum.algebra@icloud.com}
\urladdr{http://www.algebra.at}

\thanks{Research
of the second author
supported by the
Austrian Science Fund (FWF): projects
J0497-PHY, P26008-N25, and by the European Union Research Executive Agency (EUREA): project Horizon Europe 2021--2027.}


\date{Thursday, 30 April 2026}


\begin{abstract}
For quadratic fields \(k=\mathbb{Q}(\sqrt{d})\)
with discriminant \(d\),
\(3\)-class group \(\mathrm{Cl}_3(k)\simeq (\mathbb{Z}/3\mathbb{Z})^2\),
and four \textit{simple} \(3\)-principalization types
\(\varkappa(k)\in\lbrace (1122),(3122),(1231),(2231)\rbrace\),
we establish necessary and sufficient conditions
for the Galois group
\(S=\mathrm{Gal}(\mathrm{F}_3^\infty(k)/k)\)
of the unramified Hilbert \(3\)-class field tower of \(k\)
to coincide with the Galois group
\(M=\mathrm{Gal}(\mathrm{F}_3^2(k)/k)\)
of the maximal metabelian unramified \(3\)-extension of \(k\).
In the case of non-coincidence,
we study the path between \(M\) and \(S\)
in the descendant tree of the
elementary bicyclic \(3\)-group \((\mathbb{Z}/3\mathbb{Z})^2\).
For two \textit{complex} \(3\)-principalization types
\(\varkappa(k)\in\lbrace (2122),(4231)\rbrace\),
we show that infinitely many non-metabelian
possible Galois groups
\(S=\mathrm{Gal}(\mathrm{F}_3^\infty(k)/k)\)
with presumably unbounded derived length \(\mathrm{dl}(S)\)
share a common metabelianization \(M=S/S^{\prime\prime}\),
whence only partial criteria can be stated.
Minimal discriminants \(d>0\)
with assigned simple \(3\)-principalization type \(\varkappa(k)\)
and fixed length \(\ell_3(k)\in\lbrace 2,3\rbrace\)
of the \(3\)-class field tower
are determined experimentally
for nilpotency class \(\mathrm{cl}(M)\in\lbrace 5,7,9,11\rbrace\)
under assumption of the generalized Riemann hypothesis.
\end{abstract}

\maketitle


\section{Our intention}
\label{s:Intention}

\noindent
We use \textit{logarithmic abelian type invariants of second order}
to state necessary and sufficient criteria
for finite towers of successive maximal unramified abelian 3-extensions,
so-called Hilbert \(3\)-class field towers,
over quadratic number fields
to possess a length of precisely two respectively three stages.
The underlying quadratic fields
have an elementary bicyclic \(3\)-class group
and one of four well-known \textit{simple} \(3\)-principalization types
in their four unramified cyclic cubic extensions.
In two main theorems, we state the outstanding new criteria,
and we identify the graph theoretic position of crucial Galois groups. 
By means of challenging computations, our theory is underpinned by
\textit{numerical prototypes} of the four simple \(3\)-principalization types and both tower lengths.
For two \textit{complex} \(3\)-principalization types,
we can only state theorems with partial conditions.


\section{Introduction}
\label{s:Intro}

\noindent
Let \(k=\mathbb{Q}(\sqrt{d})\) be a quadratic number field
with fundamental discriminant \(d\),
elementary bicyclic \(3\)-class group \(\mathrm{Cl}_3(k)\simeq (\mathbb{Z}/3\mathbb{Z})^2\),
and one of six \(3\)-principalization types
\cite{Ma1991,Ma2012b}, \cite[p. 91]{Ma2018},
also known as
capitulation types or \textit{transfer kernel types} (TKT),

\begin{equation}
\label{eqn:SimpleTKTComplexTKT}
\begin{aligned}
& \text{E.6: }\varkappa(k)\sim(1122), \qquad 
& \text{G.16: }\varkappa(k)\sim(4231), \\
& \text{E.8: }\varkappa(k)\sim(1231), \qquad 
& \text{H.4: }\varkappa(k)\sim(2122), \\
& \text{E.9: }\varkappa(k)\sim(2231), \qquad 
& \text{two \textbf{complex} types.} \\
& \text{E.14: }\varkappa(k)\sim(3122), \\
& \text{four \textbf{simple} types,}
\end{aligned}
\end{equation}

\noindent
For the simple types,
our goal is a necessary and sufficient criterion for the coincidence
of the \(\sigma\)-automorphism group
\(S=\mathrm{Gal}(\mathrm{F}_3^\infty(k)/k)\)
of the maximal unramified pro-\(3\)-extension of \(k\)
and the \(\sigma\)-Galois group
\(M=\mathrm{Gal}(\mathrm{F}_3^2(k)/k)\)
of the maximal metabelian unramified \(3\)-extension of \(k\),
independently of the so-called \textit{state},
which is given in terms of the logarithmic \textit{abelian type invariants} (ATI)
\(\alpha(k)=(\alpha_i)_{i=1}^4\) of the \(3\)-class groups \(\alpha_i:=\mathrm{ATI}(\mathrm{Cl}_3(E_i))\)
of the four unramified cyclic cubic extensions \(E_i/k\), \(1\le i\le 4\).
When \(S=M\), then the \(3\)-class field tower of \(k\)
has \textit{exactly two stages}, \(\ell_3(k)=2\).
In the case of non-coincidence, \(S\ne M\),
we have \textit{precisely three stages}, \(\ell_3(k)=3\), for simple types,
but we have \textit{three or more stages}, \(\ell_3(k)\ge 3\), for complex types.
Then we determine the path between \(M\) and \(S\)
in the descendant tree of
either the root \(Q=\mathrm{SmallGroup}(729,49)\), briefly \(\langle 729,49\rangle\), for the types E.6, E.14 and H.4,
or the root \(U=\mathrm{SmallGroup}(729,54)\), briefly \(\langle 729,54\rangle\), for the types E.8, E.9 and G.16
\cite{BEO2005}
(the designations \textit{non-CF group} \(Q\) and \(U\) are due to Ascione et al.
\cite[Tbl. 1, p. 265, Tbl. 2, p. 266, Tbl. 3, p. 268]{AHL1977}).
We briefly speak about the \textit{Q-tree} \(\mathcal{T}(Q)\)
and the \textit{U-tree} \(\mathcal{T}(U)\).

The present article analyzes Hilbert \(3\)-class field towers
with a length \(\ell_3(k)\in\lbrace 2,3\rbrace\) for simple types,
presumably unbounded length \(\ell_3(k)\ge 2\) for complex types, and
a \textit{periodic} group \(M\) of fixed coclass \(\mathrm{cc}(M)=2\),
called \textit{second maximal class} in
\cite{AHL1977}, and
\textbf{odd} nilpotency class \(\mathrm{cl}(M)\ge 5\) for simple types,
\textbf{even} nilpotency class \(\mathrm{cl}(M)\ge 6\) for complex types.
The metabelianization \(M=S/S^{\prime\prime}\) of \(S\) is located
on one of two coclass trees with roots
\(\langle 243,6\rangle\), the parent of \(\langle 729,49\rangle\), respectively
\(\langle 243,8\rangle\), the parent of \(\langle 729,54\rangle\),
in the form of vertices with depth
\(\mathrm{dp}(M)=1\) for simple types, and
\(\mathrm{dp}(M)=2\) for complex types.
The \(n\)-th \textit{state} of \(M\) is characterized by the ATI

\begin{equation}
\label{eqn:States}
\alpha(k)=
\begin{cases}
\lbrack(n+3,n+2),1^3,(21)^2\rbrack
\text{ for the types E.6, E.14 and H.4, and} \\
\lbrack(n+3,n+2),(21)^3\rbrack
\text{ for the types E.8, E.9 and G.16},
\end{cases}
\end{equation}

\noindent
with \textbf{hetero}cyclic polarization \((n+3,n+2)\),
or also by
\textbf{odd} class \(\mathrm{cl}(M)=5+2n\) for simple types
and \textbf{even} class \(\mathrm{cl}(M)=6+2n\) for complex types,
for each integer
\(n\ge 0\), called the \textit{state parameter}.
In the case of non-coincidence, \(S\ne M\),
we emphasize the remarkable fact that
for each of these two descendant trees,
there exists a non-metabelian \textit{infinite path}
with strictly alternating step sizes \(s=2\) and \(s=1\)
and vertices of the \textit{skeleton} type
\cite[p. 91]{Ma2018}
of the mainlines of the coclass trees, that is,
c.18: \(\varkappa(k)\sim(0122)\), respectively
c.21: \(\varkappa(k)\sim(0231)\),
\cite{Ma2015c},
beginning at the bifurcation \(Q=\langle 729,49\rangle\), respectively \(U=\langle 729,54\rangle\),
from which the Schur \(\sigma\)-groups \(S\)
\cite{Ag1998}
of the \textit{ground} state \(n=0\) and of all \textit{excited} states \(n\ge 1\)
branch off as terminal leaves with step size \(s=2\)
and depth
\(\mathrm{dp}(S)=1\) for simple types and odd
\(\mathrm{dp}(S)\ge 3\) for complex types.

We begin with arithmetic foundations in \S\S\
\ref{s:Notation},
\ref{s:Artin}
and group theoretic techniques in \S\
\ref{s:GroupTheory}.
The main results and proofs are contained in \S\
\ref{s:Length}
and \S\
\ref{s:RecallImaginary},
ostensively visualized by tree diagrams in \S\
\ref{s:Diagrams},
and underpinned by computer experiments in \S\
\ref{s:Computations}
and prototypes in \S\
\ref{s:ProtoTypes}.


\section{Terminology and notation}
\label{s:Notation}

\noindent
Let \(K/\mathbb{Q}\) be an algebraic number field.
Given a prime number \(p\in\mathbb{P}\),
we denote the Sylow \(p\)-subgroup \(\mathrm{Syl}_p\mathrm{Cl}(K)\)
of the ideal class group \(\mathrm{Cl}(K)\) of \(K\) by
\(\mathrm{Cl}_p(K)\),
and we call it the \textit{\(p\)-class group} of \(K\).

\begin{definition}
\label{dfn:Hilbert}
The maximal unramified \(p\)-extension \(E\) of \(K\)
with abelian Galois group \(\mathrm{Gal}(E/K)\)
is called the \textit{Hilbert \(p\)-class field} of \(K\),
and is denoted by \(\mathrm{F}_p^1(K)\).
\end{definition}

\noindent
The concept \textit{\(p\)-class field} was coined by Hilbert in his \lq\lq Zahlbericht\rq\rq\
\cite[Satz 94, p. 279]{Hi1897},
since Hilbert conjectured an isomorphism
of the automorphism group \(\mathrm{Gal}(\mathrm{F}_p^1(K)/K)\)
to the \(p\)-class group \(\mathrm{Cl}_p(K)\) of \(K\).
His hypothesis turned out to be a special case of the Artin reciprocity law
\cite[Allgemeines Reziprozit\"atsgesetz, p. 361]{Ar1927}.

The construction of Hilbert \(p\)-class fields may be iterated and
leads to the maximal unramified pro-\(p\)-extension \(T\) of \(K\)
with a potentially infinite topological pro-\(p\) Galois group \(\mathrm{Gal}(T/K)\),
endowed with the Krull topology.


\begin{definition}
\label{dfn:Tower}
For any non-negative integer \(q\in\mathbb{N}_0\),
the \textit{\(q\)-th \(p\)-class field} of \(K\)
is defined recursively by
\(\mathrm{F}_p^0(K):=K\) for \(q=0\), and
\(\mathrm{F}_p^q(K):=\mathrm{F}_p^1(\mathrm{F}_p^{q-1}(K))\) for \(q\ge 1\).

The ascending tower
\begin{equation}
\label{eqn:Tower}
K=\mathrm{F}_p^0(K)\le\mathrm{F}_p^1(K)\le\mathrm{F}_p^2(K)\le\ldots\le\mathrm{F}_p^q(K)\le\ldots\le\mathrm{F}_p^\infty(K)
\end{equation}
of successive Hilbert \(p\)-class fields
is called the \textit{(unramified) Hilbert \(p\)-class field tower} of \(K\),
and its union is denoted by \(\mathrm{F}_p^\infty(K):=\bigcup_{q=0}^\infty\,\mathrm{F}_p^n(K)\).

If the tower becomes stationary
with \(\mathrm{F}_p^q(K)=\mathrm{F}_p^{q+1}(K)\) for some \(q\in\mathbb{N}_0\),
then the non-negative integer \(\ell_p(K):=\min\lbrace q\in\mathbb{N}_0\mid \mathrm{F}_p^q(K)=\mathrm{F}_p^{q+1}(K)\rbrace\in\mathbb{N}_0\)
is called the \textit{length} of the \(p\)-class field tower of \(K\),
otherwise the length \(\ell_p(K):=\infty\) is unbounded.
\end{definition}


\noindent
Generally, if \(\ell_p(K)\ge q\),
then the Galois group \(G=\mathrm{Gal}(\mathrm{F}_p^q(K)/K)\)
of the \(q\)-th \(p\)-class field of \(K\)
has soluble length (or derived length) \(\mathrm{sl}(G)=q\) (or \(\mathrm{dl}(G)=q\)).
In particular:

\begin{definition}
\label{dfn:Metabelian}
For \(\ell_p(K)\ge 2\),
the \textit{second Hilbert \(p\)-class field} \(\mathrm{F}_p^2(K)\)
is the maximal unramified \(p\)-extension of \(K\)
with \textit{metabelian} Galois group \(M:=\mathrm{Gal}(\mathrm{F}_p^2(K)/K)\),
since a group \(M\) with soluble length \(\mathrm{sl}(M)=2\)
is dubbed \textit{metabelian} because its commutator subgroup \(M^\prime\) is abelian.
\(M\) is called the \textit{second \(p\)-class group} of \(K\).

As opposed,
we call the topological pro-\(p\) Galois group
\(S:=\mathrm{Gal}(\mathrm{F}_p^\infty(K)/K)\)
of the maximal unramified pro-\(p\)-extension \(\mathrm{F}_p^\infty(K)\) of \(K\)
the \textit{(full \(p\)-class field) tower group} of \(K\).
The group \(M=S/S^{\prime\prime}\) is the metabelianization,
i.e., the second derived quotient, of the full tower group \(S\).
\end{definition}

\noindent
In the present article,
our focus will be on criteria for the distinction between
metabelian two-stage towers with \(\mathrm{F}_p^2(K)=\mathrm{F}_p^\infty(K)\),
\(M=S\), and \(\ell_p(K)=2\),
and on the other hand non-metabelian towers with length \(\ell_p(K)\ge 3\),
in the situation with the particular prime \(p=3\).


\section{Artin pattern and invariants of second order}
\label{s:Artin}

\noindent
In the Introduction \S\
\ref{s:Intro},
we used \textit{transfer kernel types} (TKT) \(\varkappa(k)\) and
logarithmic \textit{abelian type invariants} (ATI) \(\alpha(k)\)
of quadratic fields \(k=\mathbb{Q}(\sqrt{d})\)
to describe the goals of this article.
Now we give precise definitions of these concepts
for a quadratic number field \(k\)
with fundamental discriminant \(d(k)=d\) and
elementary bicyclic \(3\)-class group
\(\mathrm{Cl}_3(k)\simeq (\mathbb{Z}/3\mathbb{Z})^2\).
Such a field has four unramified cyclic cubic extensions
\(E_i/k\), \(1\le i\le 4\).
A natural order of these four extensions \(E_i\),
which are dihedral of degree six over \(\mathbb{Q}\)
and share a common discriminant \(d(E_i)=d^3\),
is given by increasing regulators \(R_1<\ldots<R_4\) of the
non-Galois absolutely cubic subfields \(L_i<E_i\).

\begin{definition}
\label{dfn:TKT}
Let \(T_i:\,\mathrm{Cl}_3(k)\to\mathrm{Cl}_3(E_i)\)
be the \textit{transfer} of \(3\)-classes from \(k\) to \(E_i\),
which is also called the class extension homomorphism
\(\mathfrak{a}\mathcal{P}_k\mapsto(\mathfrak{a}\mathcal{O}_{E_i})\mathcal{P}_{E_i}\).
According to Theorem 94 by Hilbert
\cite[p. 279]{Hi1897},
\(T_i\) is not injective.
Therefore, let \(H_j\) be the cyclic subgroup of order \(3\) of
\(\mathrm{Cl}_3(k)\simeq (\mathbb{Z}/3\mathbb{Z})^2\)
which corresponds to \(E_j\)
by Artin's reciprocity law
\cite{Ar1927}, 
for each \(1\le j\le 4\),
that is, \(H_j=N_{E_j/k}\mathrm{Cl}_3(E_j)\) are the norm class groups.
Then the \textit{transfer kernel type} (TKT) of \(k\),
\(\varkappa(k)=(\varkappa_i)_{i=1}^4\), (see
\cite{BuMa2015,Ma2012b}
for a permutation-invariant version as \(S_4\)-orbit)
is defined by

\begin{equation}
\label{eqn:TKT}
\varkappa_i:=
\begin{cases}
0 & \text{ if }\ker(T_i)=\mathrm{Cl}_3(k), \\
j & \text{ if }\ker(T_i)=H_j.
\end{cases}
\end{equation}
\end{definition}


\begin{definition}
\label{dfn:ATI}
The \textit{abelian type invariants} (ATI) of \(k\),
\(\alpha(k)=(\alpha_i)_{i=1}^4\),
consist of the logarithmic abelian type invariants
\(\alpha_i:=\mathrm{ATI}(\mathrm{Cl}_3(E_i))\)
of the \(3\)-class groups 
of the four unramified cyclic cubic extensions \(E_i/k\), \(1\le i\le 4\).
(In
\cite{Ma2014},
the ATI were called the \textit{transfer target type} (TTT), \(\tau(k)\).)
The \textit{Artin pattern} of \(k\) is the pair
\(\mathrm{AP}(k):=(\varkappa(k),\alpha(k))\).
\end{definition}

\noindent
The logarithmic notation is very compact, for instance,
\(21\hat{=}(9,3)\hat{=}(\mathbb{Z}/9\mathbb{Z})\times(\mathbb{Z}/3\mathbb{Z})\),
\(1^3\hat{=}(3,3,3)\hat{=}\) \((\mathbb{Z}/3\mathbb{Z})^3\), and
\((n+3,n+2)\hat{=}(3^{n+3},3^{n+2})\hat{=}(\mathbb{Z}/3^{n+3}\mathbb{Z})\times(\mathbb{Z}/3^{n+2}\mathbb{Z})\),
for any state parameter \(n\ge 0\).


\begin{definition}
\label{dfn:ATI2}
The \textit{abelian type invariants of second order} (ATI2) of \(k\),
\(\alpha^{(2)}(k)=(\alpha(E_i))_{i=1}^4\),
form a quartet where each component consists of
the logarithmic abelian type invariants (ATI)

\begin{equation}
\label{eqn:ATI2}
\alpha(E_i):=\lbrack\quad\mathrm{ATI}(\mathrm{Cl}_3(E_i));\quad\mathrm{ATI}(\mathrm{Cl}_3(E_{i,1})),\ldots,\mathrm{ATI}(\mathrm{Cl}_3(E_{i,n_i}))\quad\rbrack
\end{equation}

\noindent
of the \(3\)-class group
of the unramified cyclic cubic extension \(E_i/k\),
and of the \(3\)-class groups
of the unramified cyclic cubic extensions \(E_{i,j}/E_i\), \(1\le j\le n_i\),
where \(n_i=4\) if \(\mathrm{Cl}_3(E_i)\) is bicyclic,
and \(n_i=13\) if \(\mathrm{Cl}_3(E_i)\) is (elementary) tricyclic.
Note that \(\lbrack E_{i,j}:k\rbrack=9\) and \(\lbrack E_{i,j}:\mathbb{Q}\rbrack=18\).
\end{definition}


\section{Decisions by means of group theoretic constraints}
\label{s:GroupTheory}

\subsection{Galois action}
\label{ss:Sigma}

\noindent
The decision, whether a finite metabelian \(p\)-group \(M\)
is admissible as the second \(p\)-class group \(M=\mathrm{Gal}(\mathrm{F}_p^2(K)/K)\)
of a \textit{quadratic} number field \(K=\mathbb{Q}(\sqrt{d})\), for some \textit{odd} prime \(p\), or not,
can be made with the aid of group theoretic requirements:
it is necessary that there exists an involutory automorphism \(\sigma\in\mathrm{Aut}(M)\)
of order \(\mathrm{ord}(\sigma)=2\) which acts as the inversion \(x\mapsto x^{-1}\)
on the first and second cohomology group \(H^i(M,\mathbb{F}_p)\), for \(i\in\lbrace 1,2\rbrace\).
The same conditions are required for any pro-\(p\)-group \(S\)
to be admissible as the full \(p\)-class field tower group \(S=\mathrm{Gal}(\mathrm{F}_p^\infty(K)/K)\)
\cite[Lem. (4.1) and (4.2), p. 217]{Sf1986}.

\begin{definition}
\label{dfn:Sigma}
A pro-\(p\)-group with this property is called \textit{\(\sigma\)-group}
or \textit{group with \(\sigma\)-automophism}.
\end{definition}


\subsection{Relation rank}
\label{ss:Relators}

\noindent
Let \(p\) be a prime number
and let \(G\) be a pro-\(p\)-group.
The dimension of cohomology groups of \(G\)
acting on the finite field \(\mathbb{F}_p\) as a \(G\)-module
is crucial for applications in class field theory.

\begin{definition}
\label{dfn:Relators}
The \textit{generator rank} of \(G\) is the dimension
\(d_1(G)=\dim_{\mathbb{F}_p}(H^1(G,\mathbb{F}_p))\), and
the \textit{relation rank} of \(G\) is the dimension
\(d_2(G)=\dim_{\mathbb{F}_p}(H^2(G,\mathbb{F}_p))\).
\end{definition}


\noindent
In 1964, Shafarevich
\cite{Sh1964}
determined bounds for the relation rank
of the automorphism group \(G=\mathrm{Gal}(\mathrm{F}_{p,\mathcal{S}}^\infty(K)/K)\)
of the maximal pro-\(p\)-extension \(\mathrm{F}_{p,\mathcal{S}}^\infty(K)\) of an algebraic number field \(K\)
which is unramified outside of an assigned finite set \(\mathcal{S}\) of (real archimedean or non-archimedean) places.
In the special case of the maximal \textit{(everywhere) unramified} pro-\(p\)-extension \(\mathrm{F}_{p}^\infty(K)\) with \(\mathcal{S}=\emptyset\),
the \textbf{Shafarevich theorem} states the following sharp estimate.

\begin{theorem}
\label{thm:Unramified}
If the signature of \(K\) is \((r_1,r_2)\) (and thus \(r_1+2r_2=\lbrack K:\mathbb{Q}\rbrack\)),
the torsionfree Dirichlet unit rank of \(K\) is \(r=r_1+r_2-1\),
and \(\theta\in\lbrace 0,1\rbrace\) is an indicator for
the existence of a primitive \(p\)-th root of unity in \(K\), then
the Shafarevich bounds for the relation rank of the \(p\)-class field tower group
\(S=\mathrm{Gal}(\mathrm{F}_p^\infty(K)/K)\) are given by

\begin{equation}
\label{eqn:Unramified}
d_1(S)\le d_2(S)\le d_1(S)+r+\theta.
\end{equation}
\end{theorem}

\begin{proof}
There is a misprint in
\cite[Formula \((18^\prime)\)]{Sh1964}
which was corrected in
\cite[Thm. 5.1, p. 28]{Ma2015c}.
For \(\mathcal{S}=\emptyset\), the correction immediately yields the bounds in Formula
\eqref{eqn:Unramified}.
\end{proof}


\begin{definition}
\label{dfn:Schur}
A pro-\(p\)-group \(G\) is called a \textit{Schur \(\sigma\)-group},
or a group possessing a \textit{balanced presentation},
if \(d_2(G)=d_1(G)\),
and it is called a \textit{Schur\(+1\) \(\sigma\)-group}
if \(d_2(G)=d_1(G)+1\).
\end{definition}

\begin{corollary}
\label{cor:Quadratic}
Let \(p\) be an odd prime.
Then the \(p\)-class field tower group \(S=\mathrm{Gal}(\mathrm{F}_p^\infty(k)/k)\)
of a quadratic number field \(k=\mathbb{Q}(\sqrt{d})\)
must be a Schur \(\sigma\)-group,
when \(k\) is imaginary quadratic (with negative discriminant \(d<0\)),
and it may be a Schur\(+1\) \(\sigma\)-group or a Schur \(\sigma\)-group,
when \(k\) is real (with positive discriminant \(d>0\)).
\end{corollary}

\begin{proof}
Only for \(p=3\), the primitive \(p\)-th roots of unity
are contained in an imaginary quadratic field,
namely \(k=\mathbb{Q}(\sqrt{-3})\), but this field has
class number one and thus a trivial \(3\)-class field tower.
An imaginary quadratic field has the signature \((r_1,r_2)=(0,1)\),
its unit rank is \(r=0+1-1=0\), and the Shafarevich bounds become tight
\(d_1(S)\le d_2(S)\le d_1(S)+0+0=d_1(S)\), i. e.,
\(S\) has a balanced presentation with \(d_2(S)=d_1(S)\).
However, a real quadratic field,
which certainly cannot contain an imaginary primitive \(p\)-th root of unity for an odd prime \(p\),
has the signature \((r_1,r_2)=(2,0)\), unit rank \(r=2+0-1=1\) and the Shafarevich bounds are
\(d_1(S)\le d_2(S)\le d_1(S)+1+0=d_1(S)+1\), i. e.,
\(S\) may be a Schur\(+1\) \(\sigma\)-group or a Schur \(\sigma\)-group.
\end{proof}


\section{Finite bounds or potential infinitude ?}
\label{s:Infinity}

\noindent
There arises a basic partially \textit{mathematical} and partially \textit{philosophical} question: \\
Which of our results may be claimed \textit{for any value} of the state parameter \(n\) ?

\noindent
\(\bullet\) It is clear that our \textit{experimental} information about
concrete numerical realizations of our theoretical claims are very limited,
although they are based on the most extensive database
of relevant \textit{arithmetical} invariants for quadratic fields
\(k=\mathbb{Q}(\sqrt{d})\)
available currently.
It is due to Michael Raymond Bush in 2015, see subsection \S\ 
\ref{ss:Bush}. 
The \textit{highest excited state} for which we know a single numerical example
\(d=705\,576\,037\) of complex type H.4 is \(n=3\) (ES3).
The biggest example
\(d=336\,698\,284\) of a simple type E.14 has only \(n=2\) (ES2).
The characteristic \textit{polarizations} \(\alpha_1(k)\) of the states are
the \textbf{hetero}cyclic \(3\)-groups
\((27,9)\) for \(n=0\) (GS),
\((81,27)\) for \(n=1\) (ES1),
\((243,81)\) for \(n=2\) (ES2), and
\((729,243)\) for \(n=3\) (ES3),
whereas no real quadratic examples are known with
\((2187,729)\) for \(n=4\) (ES4).
So the bound \(n\le 3\) for the state would be sufficient
to explain all available experimental data
concerning prototypes and other paradigms in \S\
\ref{s:ProtoTypes}.

\noindent
\(\bullet\)
Our \textit{number theoretic} statements are
derived from \textit{group theoretic} theorems
by means of the Artin reciprocity law
\cite{Ar1927}.
So the former are subject to the same bounds as the latter.

\noindent
\(\bullet\)
By \textit{top-down techniques} we understand the construction of
finite \(3\)-quotients \(\mathcal{L}/\mathcal{U}\) of \textit{infinite limit groups}
\(\mathcal{U}\triangleleft\mathcal{L}\),
or also the \textit{parametrized pc-presentation} (power-commutator presentation)
of an infinite sequence \((M_n^i)_{n\ge 0}\) of finite \(3\)-groups.
Due to internal limitations of the maximal word-length in Magma
\cite{MAGMA2026},
the bound \(n\le 7\) cannot be surpassed
\cite[Rmk 3.1, p. 95]{Ma2018}.

\noindent
\(\bullet\)
By \textit{bottom-up techniques} we understand the construction of
descendants \(D=P-\#s;i\) (with step-size \(s\) and counter \(i\))
of finite \(3\)-groups \(P=\pi(D)\) (with parent-projection \(\pi\))
by means of the \(p\)-\textit{group generation} algorithm (or \textit{extension} algorithm)
by Newman and O'Brien
\cite{HEO2005, Nm1975, Ob1990}.
Here, no limitations of storage capacity (RAM) and no bounds for data types
arise in Magma
\cite{MAGMA2026}.
The amount of required CPU time, however, even on fast processors,
suggests reasonable bounds
\(n\le 19\) for difficult cases or at most \(n\le 29\) for easy cases (simple types)
\cite[Rmk 3.2, p. 96]{Ma2018}.

\noindent
\(\bullet\)
The connection between the \textit{state parameter} \(n\) and
logarithmic order (lo), nilpotency class (cl) and coclass (cc)
is given in Table
\ref{tbl:Invariants},
for simple and complex types separately.


\renewcommand{\arraystretch}{1.1}

\begin{table}[ht]
\caption{Connection between bounds for the state \(n\) and group invariants}
\label{tbl:Invariants}
\begin{center}
\begin{tabular}{|r||r|r|r||r|r|r||c|c|}
\hline
 State & \multicolumn{3}{|c||}{Simple} & \multicolumn{3}{|c||}{Complex} & Context & Theory of \\
 \(n\) & lo & cl & cc                  & lo & cl & cc                   & & \\
\hline
     2 & 14 &  9 &  5                  & 17 & 11 &  6                   & Prototypes, & Numbers \\
     3 & 17 & 11 &  6                  & 20 & 13 &  7                   & paradigms & \\
\hline
     6 & 26 & 17 &  9                  & 29 & 19 & 10                   & Top-down & Groups \\
     7 & 29 & 19 & 10                  & 32 & 21 & 11                   & methods & \\
\hline
    18 & 62 & 41 & 21                  & 65 & 43 & 22                   & difficult & Groups \\
    19 & 65 & 43 & 22                  & 68 & 45 & 23                   & Bottom-up & \\
\hline
    28 & 92 & 61 & 31                  & 95 & 63 & 32                   & easy & Groups \\
    29 & 95 & 63 & 32                  & 98 & 65 & 33                   & Bottom-up & \\
\hline
\end{tabular}
\end{center}
\end{table}


\noindent
\(\bullet\)
In spite of the finite bounds which are necessary for computations with
computer algebra systems like Magma
\cite{MAGMA2026},
it is our philosophical conviction that knowledge of
either an infinite limit group or a parametrized pc-presentation
\textit{justifies statements involving potential infinitude}. 

\noindent
Based on this opinion, we now state our main theorems in section \S\
\ref{s:Length}
and we give proofs in \S\
\ref{s:Proofs}.


\section{Length of \(3\)-class field towers}
\label{s:Length}

\subsection{Length for simple types}
\label{ss:SimpleLength}

\noindent
There is a remarkable difference between the
\textit{simple} types 
E.6: \(\varkappa(k)\sim(1122)\),
E.8: \(\varkappa(k)\sim(1231)\),
E.9: \(\varkappa(k)\sim(2231)\),
E.14: \(\varkappa(k)\sim(3122)\),
in section E of the paper
\cite[p. 36]{SoTa1934}
by Scholz and Taussky, and the
\textit{complex} types 
G.16: \(\varkappa(k)\sim(4231)\),
H.4: \(\varkappa(k)\sim(2122)\)
\cite[pp. 36--38]{SoTa1934}.
For simple types,
a non-metabelian \(3\)-class field tower
has \textit{precisely} three stages, \(\ell_3(k)=3\),
whereas for complex types, 
the length \(\ell_3(k)\ge 3\) is unbounded. 

\begin{theorem}
\label{thm:SimpleStageCriterion}
For \textbf{simple} types,
the length \(\ell_3(k)\) of the Hilbert \(3\)-class field tower of \(k\) is given as follows.
For an imaginary quadratic field \(k\) with \(d<0\), there are always three stages, \(\ell_3(k)=3\).
For a real quadratic field \(k\) with \(d>0\),
the \textbf{abelian type invariants of second order} \(\alpha^{(2)}(k)\) are required for the distinction:
for the \textbf{simple types} E.6: \((1122)\) and E.14: \((3122)\), we have two stages, \(\ell_3(k)=2\), if and only if

\begin{equation}
\label{eqn:TwoStagesQ}
\alpha^{(2)}(k)=(\overbrace{\lbrack(n+3,n+2);\alpha_0,(n+3,n+1,1)^3\rbrack}^{\text{Heterocyclic Polarization}},
\quad\overbrace{\lbrack 1^3;\alpha_0,\mathbf{(1^3)^3},(1^2)^9\rbrack,\quad\lbrack 21;\alpha_0,\mathbf{(21)^3}\rbrack^2}^{\text{Stabilization}}),
\end{equation}

\noindent
and we have three stages, \(\ell_3(k)=3\), if and only if

\begin{equation}
\label{eqn:ThreeStagesQ}
\alpha^{(2)}(k)=(\overbrace{\lbrack(n+3,n+2);\alpha_0,(n+3,n+1,1)^3\rbrack}^{\text{Heterocyclic Polarization}},
\quad\overbrace{\lbrack 1^3;\alpha_0,\mathbf{(21^2)^3},(1^2)^9\rbrack,\quad\lbrack 21;\alpha_0,\mathbf{(31)^3}\rbrack^2}^{\text{Stabilization}});
\end{equation}

\noindent
for the \textbf{simple types} E.8: \((1231)\) and E.9: \((2231)\), we have two stages, \(\ell_3(k)=2\), if and only if

\begin{equation}
\label{eqn:TwoStagesU}
\alpha^{(2)}(k)=(\overbrace{\lbrack(n+3,n+2);\alpha_0,(n+3,n+1,1)^3\rbrack}^{\text{Heterocyclic Polarization}},
\quad\overbrace{\lbrack 21;\alpha_0,\mathbf{(21)^3}\rbrack^3}^{\text{Stabilization}}),
\end{equation}

\noindent
and we have three stages, \(\ell_3(k)=3\), if and only if

\begin{equation}
\label{eqn:ThreeStagesU}
\alpha^{(2)}(k)=(\overbrace{\lbrack(n+3,n+2);\alpha_0,(n+3,n+1,1)^3\rbrack}^{\text{Heterocyclic Polarization}},
\quad\overbrace{\lbrack 21;\alpha_0,\mathbf{(31)^3}\rbrack^3}^{\text{Stabilization}});
\end{equation}

\noindent
\textbf{for each state}, parametrized by the integer \(n\ge 0\) (ground state \(n=0\), excited states \(n\ge 1\)). \\
Uniformly, the ATI of the first Hilbert \(3\)-class field \(\mathrm{F}_3^1(k)\) are denoted by \(\alpha_0:=((n+2)^2,1)\).
\end{theorem}

\noindent
Theorem
\ref{thm:SimpleStageCriterion}
only gives the connection between the length \(\ell_3(k)\) and the ATI2 \(\alpha^{(2)}(k)\).
So we need additional information on the Galois groups \(M=\mathrm{Gal}(\mathrm{F}_3^2(k)/k)\)
and \(S=\mathrm{Gal}(\mathrm{F}_3^\infty(k)/k)\)
in the case of non-coincidence, i.e., \(\ell_3(k)=3\).
After some definitions, this will be done in Theorem
\ref{thm:SimplePathStructure}.


\smallskip
\noindent
A transfer kernel type (TKT) \(\varkappa=(\varkappa_1,\ldots,\varkappa_4)\) is called
\textbf{skeleton type} if a single component \(\varkappa_i=0\) is a total capitulation, and
\textbf{hull type} if all components \(\varkappa_i\in\lbrace 1,\ldots,4\rbrace\) are partial capitulations.
For an imaginary quadratic field \(k\), the capitulation type 
\(\varkappa(k)\) must be a hull type.

\begin{definition}
\label{dfn:Cover}
By the \textbf{cover} of a finite metabelian \(3\)-group \(M\)
(terminology due to Mike F. Newman),
we understand the collection of all (isomorphism classes of) finite non-metabelian \(\sigma\)-groups \(S\),
having \(M\) isomorphic to their second derived quotient, i.e., their metabelianization,

\begin{equation}
\label{eqn:Cover}
\mathrm{cov}(M):=\lbrace S\mid\mathrm{dl}(S)\ge 3,\ S/S^{\prime\prime}\simeq M\rbrace, \ 
\mathrm{cov}_\ast(M):=\lbrace S\in\mathrm{cov}(M)\mid S \text{ is a Schur }\sigma\text{-group}\rbrace.
\end{equation}
\end{definition}

\noindent
In the sequel, we are going to use
the ANUPQ-notation of \textbf{iterated descendant groups} by \textbf{relative identifiers},
\lq\lq descendant \(=\) parent\(-\#s;j\)\rq\rq
\cite{GNO2006}
with step size \(1\le s\le\nu\), bounded by the nuclear rank \(\nu\),
and counter \(1\le j\le N_s\), bounded by the number of \(s\)-descendants \(N_s\).


\begin{theorem}
\label{thm:SimplePathStructure}
(Structure of \(3\)-class field towers with precisely three stages for \textbf{simple} types.)
There exists a unique \(\sigma\)-group \(F\), the \textbf{fork},
with nilpotency class \(\mathrm{cl}(F)=4\) and nuclear rank \(\nu(F)=2\), causing a bifurcation,
such that the path between the metabelian group \(M=S/S^{\prime\prime}\) with relation rank \(d_2(M)=3\)
and the non-metabelian Schur \(\sigma\)-group \(S\) of relation rank \(d_2(M)=2\) in the descendant tree of
\(\mathrm{SmallGroup}(9,2)=(\mathbb{Z}/3\mathbb{Z})^2\)
is given by the union of the two paths (Fig.
\ref{fig:TreeStrucSimple})

\begin{equation}
\label{eqn:SimplePath}
M=M_n^i:=F(-\#1;1-\#1;1)^n\lbrack-\#1;i\rbrack \quad
\text{ and } \quad
S=S_n^i:=F(-\#2;1-\#1;1)^n\lbrack-\#2;i\rbrack
\end{equation}

\noindent
where the parentheses \((\,)\) include vertices of skeleton type
and the brackets \(\lbrack\,\rbrack\) include vertices of simple hull type.
The final counter \(i=2,3,4\) depends on the capitulation type,
E.6 for \(i=2\), E.14 for \(i=3,4\) on the \(Q\)-tree, and
E.8 for \(i=2\), E.9 for \(i=3,4\) on the \(U\)-tree.
The exponent \(n\) of the skeleton part expresses iteration and
is the non-negative integer in the nilpotency class \(\mathrm{cl}(M)=\mathrm{cl}(S)=5+2n\),
i.e., the \textbf{state parameter}.
Finally, \(F\) is the fork between \(M\) and \(S\).
In the case of a \textbf{real} quadratic field \(k\),
the metabelian group \(M\) itself and, if \(n\ge 1\) (ES!),
additional non-metabelian Schur\(+1\) \(\sigma\)-groups
(all of them with relation rank \(d_2=3\)),
are admissible for \(S\):

\begin{equation}
\label{eqn:SimplePathReal}
S=T_{n,u}^i:=F(-\#2;1-\#1;1)^u(-\#1;1-\#1;1)^{n-u}\lbrack-\#1;i\rbrack \text{ with } 1\le u\le n, \text{ i.e.}
\end{equation}

\noindent
ostensively: if \(v:=n-u\), then the pairs \((u,v)\)
satisfy \(u+v=n\) and run through \((1,n-1),\ldots,(n,0)\).

\noindent
The cover of the second \(3\)-class group \(M_n^i\)
with \(i=2,3,4\) becomes bigger
when the state parameter \(n\ge 0\) increases.
It has \(n+1\) elements, and so it is a singleton in the ground state (GS), \(n=0\):

\begin{equation}
\label{eqn:CovMSimple}
\mathrm{cov}(M_n^i)=\lbrace T_{n,u}^i\mid 1\le u\le n\rbrace\cup\lbrace S_n^i\rbrace \quad \text{ and } \quad
\mathrm{cov}_\ast(M_n^i)=\lbrace S_n^i\rbrace.
\end{equation}
\end{theorem}


\subsection{Length for complex types}
\label{ss:ComplexLength}

\noindent
As mentioned at the beginning of the previous subsection
\S\ \ref{ss:SimpleLength},
a non-metabelian \(3\)-class field tower
has \textit{precisely} three stages, \(\ell_3(k)=3\), for simple types,
whereas the length \(\ell_3(k)\ge 3\) is \textit{unbounded} for the
\textit{complex} types 
G.16: \(\varkappa(k)\sim(4231)\),
H.4: \(\varkappa(k)\sim(2122)\)
\cite[pp. 36--38]{SoTa1934}.

\begin{theorem}
\label{thm:ComplexStageCriterion}
For \textbf{complex} types,
the length \(\ell_3(k)\) of the Hilbert \(3\)-class field tower of \(k\) is given as follows.
For an imaginary quadratic field \(k\) with \(d<0\), there are always at least three stages, \(\ell_3(k)\ge 3\).
For a real quadratic field \(k\) with \(d>0\),
the \textbf{abelian type invariants of second order} \(\alpha^{(2)}(k)\) are required for the distinction:
for the \textbf{complex type} H.4, we have the following implications:

\noindent
\(\bullet\)
If we have two stages,
\(\ell_3(k)=2\),
then

\begin{equation}
\label{eqn:TwoOrThreeStagesO}
\alpha^{(2)}(k)=(\overbrace{\lbrack(n+3,n+2);\alpha_0,(n+\mathbf{3},n+1,1)^3\rbrack}^{\text{Heterocyclic Polarization}},
\quad\overbrace{\lbrack 1^3;\alpha_0,\mathbf{(1^3)^3},(1^2)^9\rbrack,\quad\lbrack 21;\alpha_0,\mathbf{(21)^3}\rbrack^2}^{\text{Stabilization}}).
\end{equation}

\noindent
\(\bullet\)
Conversely, if Equation
\eqref{eqn:TwoOrThreeStagesO}
holds, then we may have two or three stages,
\(\ell_3(k)\in\lbrace 2,3\rbrace\).

\noindent
\(\bullet\)
However, if one of the following three conditions holds

\begin{equation}
\label{eqn:MoreStagesO}
\begin{aligned}
\alpha^{(2)}(k) &= (\overbrace{\lbrack(n+3,n+2);\alpha_0,(n+\mathbf{4},n+1,1)^3\rbrack}^{\text{Heterocyclic Polarization}},
\quad\overbrace{\lbrack 1^3;\alpha_0,\mathbf{(1^3)^3},(1^2)^9\rbrack,\quad\lbrack 21;\alpha_0,\mathbf{(21)^3}\rbrack^2}^{\text{Stabilization}}), \text{ or} \\
\alpha^{(2)}(k) &= \left(\lbrack(n+3,n+2);\alpha_0,(n+\mathbf{3},n+1,1)^3\rbrack,
\quad\lbrack 1^3;\alpha_0,\mathbf{(21^2)^3},(1^2)^9\rbrack,\quad\lbrack 21;\alpha_0,\mathbf{(31)^3}\rbrack^2\right), \text{ or} \\
\alpha^{(2)}(k) &= \left(\lbrack(n+3,n+2);\alpha_0,(n+\mathbf{4},n+1,1)^3\rbrack,
\quad\lbrack 1^3;\alpha_0,\mathbf{(21^2)^3},(1^2)^9\rbrack,\quad\lbrack 21;\alpha_0,\mathbf{(31)^3}\rbrack^2\right),
\end{aligned}
\end{equation}

\noindent
then we certainly have at least three stages, \(\ell_3(k)\ge 3\); \\
for the \textbf{complex type} G.16, we have the following implications:

\noindent
\(\bullet\)
If we have two stages,
\(\ell_3(k)=2\),
then

\begin{equation}
\label{eqn:TwoOrThreeStagesS}
\alpha^{(2)}(k)=(\overbrace{\lbrack(n+3,n+2);\alpha_0,(n+\mathbf{3},n+1,1)^3\rbrack}^{\text{Heterocyclic Polarization}},
\quad\overbrace{\lbrack 21;\alpha_0,\mathbf{(21)^3}\rbrack^3}^{\text{Stabilization}}),
\end{equation}

\noindent
\(\bullet\)
Conversely, if Equation
\eqref{eqn:TwoOrThreeStagesS}
holds, then we may have two or three stages,
\(\ell_3(k)\in\lbrace 2,3\rbrace\).

\noindent
\(\bullet\)
However, if one of the following three conditions holds

\begin{equation}
\label{eqn:MoreStagesS}
\begin{aligned}
\alpha^{(2)}(k) &= (\overbrace{\lbrack(n+3,n+2);\alpha_0,(n+\mathbf{4},n+1,1)^3\rbrack}^{\text{Heterocyclic Polarization}},
\quad\overbrace{\lbrack 21;\alpha_0,\mathbf{(21)^3}\rbrack^3}^{\text{Stabilization}}), \text{ or} \\
\alpha^{(2)}(k) &= \left(\lbrack(n+3,n+2);\alpha_0,(n+\mathbf{3},n+1,1)^3\rbrack,
\quad\lbrack 21;\alpha_0,\mathbf{(31)^3}\rbrack^3\right), \text{ or} \\
\alpha^{(2)}(k) &= \left(\lbrack(n+3,n+2);\alpha_0,(n+\mathbf{4},n+1,1)^3\rbrack,
\quad\lbrack 21;\alpha_0,\mathbf{(31)^3}\rbrack^3\right),
\end{aligned}
\end{equation}

\noindent
then we certainly have at least three stages, \(\ell_3(k)\ge 3\). \\
The claims are proven for \textbf{bounded state},
parametrized by the integer \(n\) (ground state \(n=0\), excited states \(1\le n\le 4\)),
but we conjecture they hold for any state \(n\ge 0\). \\
Uniformly, the ATI of the first Hilbert \(3\)-class field \(\mathrm{F}_3^1(k)\) are denoted by \(\alpha_0:=(n+3,n+2,1)\).
\end{theorem}


\begin{theorem}
\label{thm:ComplexPathStructure}
(Structure of \(3\)-class field towers with three or more stages for \textbf{complex} types.)
There exists a unique \(\sigma\)-group \(F\), the \textbf{fork},
with nilpotency class \(\mathrm{cl}(B)=4\) and nuclear rank \(\nu(B)=2\), causing bifurcation,
such that the path between the metabelian group \(M=S/S^{\prime\prime}\) with relation rank \(d_2(M)=3\)
and one of the \textbf{infinitely} many possible
non-metabelian Schur \(\sigma\)-groups \(S\) of relation rank \(d_2(M)=2\) 
in the descendant tree of
\(\mathrm{SmallGroup}(9,2)=(\mathbb{Z}/3\mathbb{Z})^2\)
is given by the union of the two paths (Fig.
\ref{fig:TreeStrucComplex})
\begin{equation}
\label{eqn:ComplexPath}
\begin{aligned}
M=M_n^i &= F(-\#1;1-\#1;1)^n\lbrack-\#1;i-\#1;1\rbrack \quad
\text{ and } \quad \\
S=S_{n,t}^i &= F(-\#2;1-\#1;1)^n\lbrack-\#2;i-\#1;1\rbrack\lbrack-\#2;1-\#1;1\rbrack^{t}\lbrack-\#2;2\rbrack
\end{aligned}
\end{equation}
where the parentheses \((\,)\) include vertices of skeleton type
and the brackets \(\lbrack\,\rbrack\) include vertices of complex hull type.
The intermediate counter \(i=5,6\) represents two possible variants,
which cannot be distinguished arithmetically.
The complex hull type is
H.4 on the \(Q\)-tree and
G.16 on the \(U\)-tree.
The exponent \(n\) of the skeleton part is the non-negative integer in the nilpotency class \(\mathrm{cl}(M)=6+2n\),
i.e., the \textbf{state parameter},
and the exponent \(t\) of the complex hull part is the non-negative integer in the nilpotency class \(\mathrm{cl}(S)=7+2t+2n\),
i.e., the \textbf{sub-state parameter}.
Finally, \(F\) is the fork between \(M\) and \(S\).
In the case of a \textbf{real} quadratic field \(k\),
additional Schur\(+1\) \(\sigma\)-groups are admissible for \(S\):
Firstly, as immediate descendants of \(M\) with \textbf{child topology},

\begin{equation}
\label{eqn:ComplexPathChild}
S=R_{n,j}^i:=M_n^i-\#1;j \quad \text{ with } j=1,2,3 \text{ for type H.4 }, j=1,2 \text{ for type G.16}.
\end{equation}

\noindent
Secondly, via the fork \(F\) with \textbf{fork topology}, if \(n\ge 1\),

\begin{equation}
\label{eqn:ComplexPathFork}
S=T_{n,u}^i:=F(-\#2;1-\#1;1)^u(-\#1;1-\#1;1)^{n-u}\lbrack-\#1;i-\#1;1\rbrack \text{ with } 1\le u\le n, \text{ i.e.}
\end{equation}

\noindent
ostensively: if \(v:=n-u\), then the pairs \((u,v)\)
satisfy \(u+v=n\) and run through \((1,n-1),\ldots,(n,0)\).

\noindent
The cover of the second \(3\)-class group \(M_n^i\)
with \(i=5,6\) is always infinite,
and its components becomes bigger
when the state parameter \(n\ge 0\) increases:

\begin{equation}
\label{eqn:CovMComplex}
\begin{aligned}
\mathrm{cov}(M_n^i) =
& \lbrace R_{n,j}^i\mid j=1,2(,3)\rbrace\quad\cup \\
& \bigcup\lbrace\mathcal{T}_0(T_{n,u}^i)\mid 1\le u\le n\rbrace\quad\cup \\
& \bigcup\lbrace\mathcal{T}_0(\pi(S_{n,t}^i))\setminus\mathcal{T}_0(\pi(S_{n,t+1}^i))\mid t\ge 0\rbrace
\quad \text{ and } \\
\mathrm{cov}_\ast(M_n^i) =
& \lbrace S_{n,t}^i\mid t\ge 0\rbrace,
\end{aligned}
\end{equation}

\noindent
where \(\pi\) denotes the parent-operator,
and \(\mathcal{T}_0\) is the pruned descendant tree restricted to \(\sigma\)-groups.
The \(\mathcal{T}_0(T_{n,u}^i)\) are \(n\) finite saplings whose depth \(\mathrm{dp}=3+2u\) increases with \(u\).
The differences \(\mathcal{T}_0(\pi(S_{n,t}^i))\setminus\mathcal{T}_0(\pi(S_{n,t+1}^i))\)
are successive finite sections of the infinite branch \(\mathcal{T}_0(\pi(S_{n,0}^i))\),
whose depth \(\mathrm{dp}=3+2n\) increases with \(n\).
Their vertices conjecturally possess unbounded soluble length. \\
The claims are proved for \textbf{bounded state},
parametrized by the integer \(n\) (ground state \(n=0\), excited states \(1\le n\le 4\)),
but we conjecture they hold for any state \(n\ge 0\). \\
\end{theorem}

\noindent
To reduce the complexity of the tree diagram,
Schur\(+1\) \(\sigma\)-groups are not drawn in Figure
\ref{fig:TreeStrucComplex}.
More details are drawn in Figure
\ref{fig:ComplexTopologies}
for the ground state (GS) \(n=0\):
the vertices \(R_{0,j}^i\) with child topology
\(j=1,2,3\) for type H.4 (drawn), \(j=1,2\) for type G.16,
and the finite section
\(\mathcal{T}_0(\pi(S_{00}^i))\setminus\mathcal{T}_0(\pi(S_{01}^i))\)
with depth \(\mathrm{dp}=3\)
of the infinite branch \(\mathcal{T}_0(\pi(S_{00}^i))\)
with periodic bifurcations.
The vertices \(U_{0j}^i\), \(j=1,\ldots,4\), and \(V_{0}^i\)
refer to both complex hull types,
for the terminal leaves \(W_{0j}^i\), however,
we have \(j=1,2,3\) for type H.4 (drawn), but \(j=1,2\) only for type G.16


\section{Proofs concerning the length of \(3\)-class field towers}
\label{s:Proofs}

\noindent
First we show in two Propositions
\ref{prp:TopDownQ}
and
\ref{prp:TopDownU}
how the main lines and branches down to depth \(\mathrm{dp}=1\)
of both relevant trees, the \(Q\)-tree and the \(U\)-tree,
can be given either by parametrized polycyclic pc-presentations
or as quotients of infinite limit groups.
These propositions concern the metabelianizations \(M\)
with coclass \(\mathrm{cc}(M)=2\)
of simple and complex types simultaneously.

The reason why only the \(Q\)-tree and the \(U\)-tree
are relevant for our assigned six TKTs in Formula
\eqref{eqn:SimpleTKTComplexTKT}
is given with the aid of the \lq\lq main theorem on class and coclass from IPAD\rq\rq\
and its corollary
\cite[\S\ 4, Thm. 2 and Cor. 1]{AoMa2025}
in \S\
\ref{ss:Bush}:
after elimination of the finite \(3\)-groups \(G\) of maximal class,
i.e., coclass \(\mathrm{cc}(G)=1\),
among all descendants of the elementary bicyclic \(3\)-group
\(\mathrm{SmallGroup}(9,2)=(\mathbb{Z}/3\mathbb{Z})^2\),
we arrive at the seven groups in Hall's isoclinism family \(\Phi_6\)
by bifurcation from the extra-special \(3\)-group \(\langle 27,3\rangle\).
Then the \(\varepsilon\)-invariant of the IPAD
and the required shape of the six TKTs eliminates five of these groups,
\(\langle 243,i\rangle\) with \(i=3,4,5,7,9\),
and only \(\langle 243,6\rangle\), the parent of \(Q=\langle 729,49\rangle\),
and \(\langle 243,8\rangle\), the parent of \(U=\langle 729,54\rangle\),
remain admissible.


The following result shows that
certain \(3\)-groups of class at least \(5\)
on the coclass tree \(\mathcal{T}^2(\langle 243,6\rangle)\)
belong to \(6+4=10\) periodic coclass sequences (6 odd, 4 even)
with period length \(2\).

\begin{proposition}
\label{prp:TopDownQ}
For each integer \(c\ge 5\),
there are \(6\) \textbf{metabelian} descendants \(G\) of \(\langle 243,6\rangle\),
having nilpotency class \(\mathrm{cl}(G)=c\), coclass \(\mathrm{cc}(G)=2\), and order \(\lvert G\rvert=3^{c+2}\),
with two generators \(x,y\) and parametrized pc-presentation
(where \(t_3\) and \(s_j\) belong to the abelian subgroup \(G^\prime<G\))

\begin{equation}
\label{eqn:QG}
\begin{aligned}
  G   =   & \langle\ x,y,s_2,t_3,s_3,s_4,\ldots,s_c\ \mid \\
          & s_2=\lbrack y,x\rbrack,\ t_3=\lbrack s_2,y\rbrack,\ s_j=\lbrack s_{j-1},x\rbrack \text{ for } 3\le j\le c, \\
          & s_j^3=s_{j+2}^2s_{j+3} \text{ for } 2\le j\le c-3,\ s_{c-2}^3=s_c^2,\ t_3^3=1, \\
          & R(x)=1,\ R(y)=1\ \rangle, \\
\end{aligned}
\end{equation}

\noindent
where the relators \(R(x)\) and \(R(y)\) are given by

\begin{equation}
\label{eqn:QR}
\begin{aligned}
R(x) = &
\begin{cases}
x^3         & \text{ for } G \text{ of TKT } \mathrm{c}.18 \text{ or } \mathrm{H}.4, \text{ and} \\
x^3s_c^{-1} & \text{ for } G \text{ of TKT } \mathrm{E}.6 \text{ or } \mathrm{E}.14,
\end{cases}
\\
R(y) = &
\begin{cases}
y^3s_3^{-2}s_4^{-1}         & \text{ for } G \text{ of TKT } \mathrm{c}.18 \text{ or } \mathrm{E}.6, \text{ and either} \\
y^3s_3^{-2}s_4^{-1}s_c^{-1} & \text{ or } \\
y^3s_3^{-2}s_4^{-1}s_c^{-2} & \text{ for } G \text{ of TKT } \mathrm{H}.4 \text{ or } \mathrm{E}.14.
\end{cases}
\end{aligned}
\end{equation}

\noindent
For odd class \(c\ge 5\) the \(6\) groups are pairwise non-isomorphic \(\sigma\)-groups.\\
For even class \(c\ge 6\), the two pairs of groups sharing the same TKT
(\(\mathrm{H}.4\) and \(\mathrm{E}.14\))
are isomorphic, and thus only \(4\) groups are pairwise non-isomorphic,
and only the mainline group is a \(\sigma\)-group.
\end{proposition}

\begin{proof}
The claims are a consequence of direct computations of the TKT
as kernels of Artin transfers from \(G\) to its four maximal subgroups
\cite[Alg. 3.3, pp. 104--106]{Ma2018},
depending on the relators \(R(x)\) and \(R(y)\) but
independently of the nilpotency class \(c\).
The relators have also been expressed
in terms of exponential parameters \(\alpha,\beta\),
as used by Blackburn for \(3\)-groups of maximal class, in
\cite[Formulas 13--14, p. 107]{Ma2018}.
Skeleton type is \((0122)\).
\end{proof}


\noindent
The following result shows that
certain \(3\)-groups of class at least \(5\)
on the coclass tree \(\mathcal{T}^2(\langle 243,8\rangle)\)
belong to \(6+4=10\) periodic coclass sequences (6 odd, 4 even)
with period length \(2\).

\begin{proposition}
\label{prp:TopDownU}
For each integer \(c\ge 5\),
there are \(6\) \textbf{metabelian} descendants \(G\) of \(\langle 243,8\rangle\),
having nilpotency class \(\mathrm{cl}(G)=c\), coclass \(\mathrm{cc}(G)=2\), and order \(\lvert G\rvert=3^{c+2}\),
with two generators \(x,y\) and parametrized pc-presentation
(where \(s_3\) and \(t_j\) belong to the abelian subgroup \(G^\prime<G\))

\begin{equation}
\label{eqn:UG}
\begin{aligned}
  G   =   & \langle\ x,y,t_2,s_3,t_3,t_4,\ldots,t_c\ \mid \\
          & t_2=\lbrack y,x\rbrack,\ s_3=\lbrack t_2,x\rbrack,\ t_j=\lbrack t_{j-1},y\rbrack \text{ for } 3\le j\le c, \\
          & t_j^3=t_{j+2}^2t_{j+3} \text{ for } 2\le j\le c-3,\ t_{c-2}^3=t_c^2,\ s_3^3=1, \\
          & R(y)=1,\ R(x)=1\ \rangle, \\
\end{aligned}
\end{equation}

\noindent
where the relators \(R(y)\) and \(R(x)\) are given by

\begin{equation}
\label{eqn:UR}
\begin{aligned}
R(y) = &
\begin{cases}
y^3s_3^{-1}         & \text{ for } G \text{ of TKT } \mathrm{c}.21 \text{ or } \mathrm{G}.16, \text{ and} \\
y^3s_3^{-1}t_c^{-1} & \text{ for } G \text{ of TKT } \mathrm{E}.8 \text{ or } \mathrm{E}.9,
\end{cases}
\\
R(x) = &
\begin{cases}
x^3t_3^{-1}t_4^{-2}t_5^{-1}         & \text{ for } G \text{ of TKT } \mathrm{c}.21 \text{ or } \mathrm{E}.8, \text{ and either} \\
x^3t_3^{-1}t_4^{-2}t_5^{-1}t_c^{-1} & \text{ or } \\
x^3t_3^{-1}t_4^{-2}t_5^{-1}t_c^{-2} & \text{ for } G \text{ of TKT } \mathrm{G}.16 \text{ or } \mathrm{E}.9.
\end{cases}
\end{aligned}
\end{equation}

\noindent
For odd class \(c\ge 5\) the \(6\) groups are pairwise non-isomorphic \(\sigma\)-groups.\\
For even class \(c\ge 6\), the two pairs of groups sharing the same TKT
(\(\mathrm{G}.16\) and \(\mathrm{E}.9\))
are isomorphic, and thus only \(4\) groups are pairwise non-isomorphic,
and only the mainline group is a \(\sigma\)-group.
\end{proposition}

\begin{proof}
The claims are a consequence of direct computations of the TKT
as kernels of Artin transfers from \(G\) to its four maximal subgroups,
depending on the relators \(R(y)\) and \(R(x)\) but
independently of the nilpotency class \(c\).
Here, the equivalent skeleton type is \((0134)\sim (0231)\).
\end{proof}

\noindent
For the special case of the \(Q\)-tree and the \(U\)-tree,
the two Propositions
\ref{prp:TopDownQ}
and
\ref{prp:TopDownU}
provide an explicit proof of
the \textit{periodicity of tree branches},
proven generally and independently in
\cite{dS2001,EkLG2008,ELNO2013,LG1994,Sv1994}.
The finitely many pre-periodic instances down to class \(c=3,4\)
have similar pc-presentations with slight degenerations,
due to the low values of \(c\).
However, our metabelianizations \(M\) in Theorem
\ref{thm:SimplePathStructure}
and
\ref{thm:ComplexPathStructure}
have odd nilpotency class \(\mathrm{cl}(M)=2n+5\ge 5\),
for each state \(n\ge 0\).


\begin{corollary}
\label{cor:Metabelianization}
For \textbf{simple} types,
the metabelian second \(3\)-class group
\(M=\mathrm{Gal}(\mathrm{F}_3^2(k)/k)\)
of the state with parameter \(n\ge 0\)
is given by 
\(M=M_n^i:=F(-\#1;1-\#1;1)^n\lbrack-\#1;i\rbrack\),
as a relative identifier
with respect to the fork \(F=Q\) or \(F=U\).
The final counter \(2\le i\le 4\) depends on the TKT,
E.6 for \(i=2\), E.14 for \(i=3,4\) on the \(Q\)-tree,
E.8 for \(i=2\), E.9 for \(i=3,4\) on the \(U\)-tree.
\end{corollary}

\begin{proof}
In a coclass tree, all edges have step size \(s=1\).
All metabelianizations \(M=M_n^i\) in the Formula
\eqref{eqn:SimplePath}
of Theorem
\ref{thm:SimplePathStructure}
have odd nilpotency class \(c=\mathrm{cl}=2n+5\),
fixed coclass \(r=\mathrm{cc}=2\),
and thus odd logarithmic order \(\mathrm{lo}=2n+7\).
Thus they belong to the \(6\) odd periodic coclass sequences.
They are terminal leaves with depth \(\mathrm{dp}=1\)
away from the mainline of the pruned coclass tree
\({}^\ast\mathcal{T}^2(F)\) in Figure
\ref{fig:TreeStrucSimple}.
Corresponding to the group \(G\) in Proposition
\ref{prp:TopDownQ},
respectively
\ref{prp:TopDownU},
with odd class \(c=2n+5\),
their relative identifiers are of the shape
\(M=M_n^i:=F(-\#1;1-\#1;1)^n\lbrack-\#1;i\rbrack\)
with state parameter \(n\ge 0\) and counter \(i=2,3,4\)
specifying the TKT, as given in the theorem.
We add that the meaning of the counter \(i=1\) is
the mainline vertex \(X_c^r\) with \(r=2\), Eqn.
\eqref{eqn:MainlineVertices}.
\end{proof}


\noindent
The complete theory of the mainlines with vertices of skeleton type,
c.18, \(\varkappa\sim (0122)\), on the \(Q\)-tree, respectively
c.21, \(\varkappa\sim (0231)\), on the \(U\)-tree,
of the descendant tree \(\mathcal{T}(R)\) with root \(R\), containing an
infinite main trunk of alternating step sizes \(s=2\) and \(s=1\),
is based on the following infinite limit groups, due to Mike F. Newman.

\begin{definition}
\label{dfn:MainlineLimit}
The \textit{skeleton group} with sign \(-1\) for the \(Q\)-tree and sign \(+1\) for the \(U\)-tree is

\begin{equation}
\label{eqn:ProfiniteGroup}
\mathcal{L}_{\mp 1}:=\langle\ a,t\ \mid\ (at)^3=a^3,\ \lbrack\lbrack t,a\rbrack,t\rbrack=a^{\mp 3}\ \rangle.
\end{equation}

\noindent
For each coclass \(r\ge 2\),
the \textit{mainline-\(r\) limit} is defined as
a quotient of \(\mathcal{L}_{\mp 1}\):

\begin{equation}
\label{eqn:MainlineLimits}
\mathcal{L}_{\mp 1}^{r}:=\mathcal{L}_{\mp 1}\ /\ \langle\ a^{3^r}\ \rangle.
\end{equation}

\noindent
Finally, for each coclass \(r\ge 2\), and for each nilpotency class \(c\ge 2r-1\),
the \textit{mainline-\(r\) vertex of class \(c\)} on the coclass-\(r\) tree \(\mathcal{T}^r(X_{2r-1}^r)\)
is defined as a quotient of \(\mathcal{L}_{\mp 1}^{(r)}\):

\begin{equation}
\label{eqn:MainlineVertices}
X_c^r:=
\begin{cases}
\mathcal{L}_{\mp 1}^{r}\ /\ \langle\ \lbrack t,a\rbrack^{3^j}\ \rangle & \text{ if }\ c=2j+1\ \text{ is odd with }\ j\ge r-1, \\
\mathcal{L}_{\mp 1}^{r}\ /\ \langle\ t^{3^j}\ \rangle                  & \text{ if }\ c=2j\ \text{ is even with }\ j\ge r.
\end{cases}
\end{equation}
\end{definition}


\begin{theorem}
\label{thm:MainLines}
(Mainline vertices as quotients of the limit group \(\mathcal{L}_{\mp}\).) 
\begin{enumerate}
\item
For each coclass \(r\ge 2\), and for each class \(c\ge 2r-1\),
the mainline vertex of coclass \(r\) and nilpotency class \(c\) in the tree \(\mathcal{T}(R)\) is isomorphic to \(X_c^r\).
\item
For each coclass \(r\ge 2\),
the projective limit of the mainline \(\left(X_c^r\right)_{c\ge 2r-1}\) with vertices of coclass \(r\)
in the tree \(\mathcal{T}(R)\) is isomorphic to \(\mathcal{L}_{\mp}^{(r)}\).
\item
\(\mathcal{L}_{\mp}\) is an infinite non-nilpotent profinite limit group.
\end{enumerate}
\end{theorem}

\begin{proof}
In
\cite[Alg. 3.1, Rmk. 3.1--2, Thm. 3.3, Cnj. 3.2, pp. 93--96]{Ma2018},
item (3) was proved generally.
By means of top-down techniques,
item (2) was proved for bounded coclass \(2\le r\le 8\), and
item (1) for bounded coclass \(2\le r\le 8\) and bounded class \(2r-1\le c\le 35\).
Since bottom-up techniques admit much bigger bounds \(r\le 32\) and \(c\le 63\),
that is a logarithmic order of \(c+r=95\),
items (1) and (2) were conjectured to hold for any coclass \(r\),
and item (1) for any class \(c\).
However, the Formulas
\eqref{eqn:ProfiniteGroup},
\eqref{eqn:MainlineLimits}, and
\eqref{eqn:MainlineVertices}
comprise parametrized formation laws with two unbounded parameters \(r\) and \(c\).
Therefore, they hold generally,
since they are independent of verifications with the aid of
necessarily limited information technology
in the form of the indicated version V2.22-7 
of the computational algebra system Magma
\cite{MAGMA2026}.
\end{proof}


\begin{definition}
\label{dfn:CoverLimit}
For \(e\in\lbrace 0,1\rbrace\) (\(e=0\) for the \(Q\)-tree and \(e=1\) for the \(U\)-tree),
we define the \textit{cover limit} (according to Mike F. Newman) to be the infinite limit group
\begin{equation}
\label{eqn:CoverLimit}
\begin{aligned}
\mathcal{C}^{(e)}
:= \langle&\ a,t,u,y,z\ \mid\ t^a=u,\ u^atuy=\lbrack u,t\rbrack^e,\ a^3\lbrack t,a,t\rbrack=z,\ \lbrack u,t,t\rbrack=\lbrack u,t,u\rbrack=1, \\
          &\ y^3=1,\ \lbrack a,y\rbrack=\lbrack t,y\rbrack=\lbrack u,y\rbrack=\lbrack z,y\rbrack=1,\ z^3 =1,\ \lbrack t,z\rbrack=\lbrack u,z\rbrack=1\ \rangle.
\end{aligned}
\end{equation}
which is related to a group in
\cite{ELNO2013}.
For each \(\ell\in\lbrace -1,0,1\rbrace\) and for each integer \(c\ge 4\), let
\begin{equation}
\label{eqn:Quotient}
\mathcal{Q}_c^{(e,\ell)}:=\mathcal{C}^{(e)}\ /\ \langle\ yw_c^\ell v_c,\ zw_c\ \rangle
\end{equation}
be the \textit{class-\(c\) quotient with parameter} \(\ell\) of \(\mathcal{C}^{(e)}\), where
\(w_c:=\lbrack t,\overbrace{a,\ldots,a}^{(c-1) \text{ times}}\rbrack\) and
\(v_c:=\lbrack w_{c-2},\lbrack t,a\rbrack\rbrack\).
\end{definition}


\noindent
The next theorem is a second crucial ingredient of this proof section,
establishing the finiteness and structure of the \textit{cover}
for each metabelian \(3\)-group \(M\) with transfer kernel type in section \(\mathrm{E}\).

\begin{theorem}
\label{thm:ExplicitCovers}
(Explicit covers of metabelian \(3\)-groups.)
Let \(M_n^i\) in the coclass-\(2\) tree \(\mathcal{T}^{2}\left(X_{3}^{2}\right)\) be
the metabelianization of odd nilpotency class \(c=2n+5\ge 5\), \(n\ge 0\),
with \textbf{simple} TKT

\[
\varkappa(k)=
\begin{cases}
(1122),\ \mathrm{E}.6, \text{ resp. } (1231),\ \mathrm{E}.8 & \text{ if } i=2, \\
(3122),\ \mathrm{E}.14, \text{ resp. } (2231),\ \mathrm{E}.9 & \text{ if } i=3 \text{ or } i=4.
\end{cases}
\]

\begin{enumerate}
\item
The \textbf{cover} of \(M_n^i\) is given by
\begin{equation}
\label{eqn:CovSecE}
\mathrm{cov}\left(M_n^i\right)=
\left\lbrace T_{n,1}^i,\ldots,T_{n,n}^i,S_n^i\right\rbrace, \quad \text{ for } n\ge 0,\ 2\le i\le 4.
\end{equation}
In particular, the cover is a finite set with \(n+1\) elements
which are \(\sigma\)-groups.
\item
The \textbf{Shafarevich cover} of \(M_n^i\) with respect to imaginary quadratic fields \(k\) is given by
\begin{equation}
\label{eqn:ShafarevichCovSecE}
\mathrm{cov}_\ast\left(M_n^i\right)=
\left\lbrace S_n^i\right\rbrace, \quad \text{ for } n\ge 0,\ 2\le i\le 4.
\end{equation}
In particular, the Shafarevich cover contains a unique non-metabelian Schur \(\sigma\)-group.
\item
The class-\(c\) quotient with parameter \(\ell\) of the cover limit \(\mathcal{C}^{(e)}\) is isomorphic to the Schur \(\sigma\)-group
\(S_{n}^{i}\simeq\mathcal{Q}^{(e,\ell)}_c\), for odd class \(c=2n+5\), \(n\ge 0\).
The precise correspondence between the parameters \(\ell\) and \(i\) is given in the following way:
\begin{equation}
\label{eqn:IsomCovSecE}
\begin{aligned}
\text{simple type }\mathrm{E}.6:\ &\mathcal{Q}^{(0,0)}_c \simeq
S_{n}^{2}\in\mathcal{T}(Q), \ell=0 \text{ corresponds to } i=2, \\
\text{simple type }\mathrm{E}.8:\ &\mathcal{Q}^{(1,0)}_c \simeq
S_{n}^{2}\in\mathcal{T}(U), \ell=0 \text{ corresponds to } i=2, \\
\text{simple type }\mathrm{E}.14:\ &\mathcal{Q}^{(0,-1)}_c\simeq
S_{n}^{3}\in\mathcal{T}(Q) \ell=-1 \text{ corresponds to } i=3, \\
\text{simple type }\mathrm{E}.9:\ &\mathcal{Q}^{(1,-1)}_c \simeq
S_{n}^{3}\in\mathcal{T}(U), \ell=-1 \text{ corresponds to } i=3, \\
\text{simple type }\mathrm{E}.9:\ &\mathcal{Q}^{(1,+1)}_c \simeq
S_{n}^{4}\in\mathcal{T}(U), \ell=+1 \text{ corresponds to } i=4.
\end{aligned}
\end{equation}
The variant \(e=0\), respectively \(e=1\), is associated to the root \(R=\langle 243,6\rangle\), respectively \(R=\langle 243,8\rangle\).
\item
A parametrized family of \textbf{fork topologies}
for second \(3\)-class groups \(M=\mathrm{Gal}\left(F_3^{2}(k)/k\right)\) of imaginary quadratic fields \(k\)
is given uniformly for the states \(n\ge 0\) (ground state, GS, for \(n=0\), excited state, ES \(n\), for \(n\ge 1\))
of \textbf{simple} transfer kernel types in section \(\mathrm{E}\)
by the following \textbf{symmetric topology symbol} for the connecting path \(P\),
\begin{equation}
\label{eqn:ForkSecE}
P=
\overbrace{\mathrm{E}\stackrel{1}{\rightarrow}}^{\text{Leaf}}\quad
\overbrace{\left\lbrace\mathrm{c}\stackrel{1}{\rightarrow}\mathrm{c}\stackrel{1}{\rightarrow}\right\rbrace^{n}\ }^{\text{Mainline}}\quad
\overbrace{\mathrm{c}}^{\text{Fork}}\quad
\overbrace{\left\lbrace\stackrel{2}{\leftarrow}\mathrm{c}\stackrel{1}{\leftarrow}\mathrm{c}\right\rbrace^{n}\ }^{\text{Maintrunk}}\quad
\overbrace{\stackrel{2}{\leftarrow}\mathrm{E}}^{\text{Leaf}}
\end{equation}
with skeleton type \(\mathrm{c}\) and the following invariants: \\
distance \(d=(2n+1)+(2n+1)=4n+2\),
weighted distance \(w=(2n+1)+(3n+2)=5n+3\), \\
class increment \(\Delta\mathrm{cl}=(2n+5)-(2n+5)=0\),
coclass increment \(\Delta\mathrm{cc}=(n+3)-2=n+1\), \\
logarithmic order increment \(\Delta\mathrm{lo}=(3n+8)-(2n+7)=n+1\).
\end{enumerate}
\end{theorem}

\begin{proof}
We compare the uniform generator rank \(d_1=2\) of all involved groups \(M_n^i\) and \(S_{n}^{i}\)
with odd class \(c=2n+5\ge 5\), coclass \(r\ge 2\), and counter \(2\le i\le 4\) specifying the simple TKT,
with their relation rank \(d_2\).
Since \(d_2=\mu\) and the \(3\)-multiplicator rank is \(\mu=2\) for
\(S_{n}^{i}\) with odd class \(c=2n+5\ge 5\) and coclass \(r=n+3\ge 3\),
but \(\mu=3\) for \(M_n^i\) with \(r=2\),
only the groups \(S_{n}^{i}\) are Schur \(\sigma\)-groups with balanced presentation \(d_2=2=d_1\),
and are therefore admissible as \(3\)-class field tower groups \(S\) of imaginary quadratic fields \(k\),
according to our corrected version, Theorem
\ref{thm:Unramified},
of the Shafarevich Theorem and Corollary
\ref{cor:Quadratic}.
Finally we remark that the nuclear rank is \(\nu=0\) for \(M_n^i\) with odd class \(c=2n+5\),
i.e., the metabelianizations \(M=M_n^i\) are terminal leaves.

Since the finite class-\(c\) quotients \(\mathcal{Q}^{(e,k)}_c\) with odd \(c\ge 5\) and parameter \(\ell=-1,0,+1\)
of the infinite cover limit \(\mathcal{C}^{(e)}\), \(e=0,1\),
contain the unlimited nilpotency class \(c\), Theorem
\ref{thm:ExplicitCovers}
can be stated without bounds,
in contrast to the careful formulation of 
\cite[Thm. 3.5 and Cnj. 3.3, pp. 98--100]{Ma2018}
with bound \(n\le 8\) for the state parameter \(n\).

We remark that the execution of Algorithm 3.2 in
\cite[pp. 97--98]{Ma2018}
experimentally proves the isomorphisms \(\mathcal{Q}^{(e,\ell)}_c\simeq S_{n}^{i}\),
with odd class \(c=2n+5\) and suitably corresponding parameters \(\ell\) and \(i\),
for \(5\le c\le 21\), that is, \(0\le n\le 8\).
The algorithm is initialized by the starting group \(R=X_3^{2}\) of coclass \(2\),
and navigates through the mainline vertices \(X_{c}^{2}\), \(c\ge 3\), of the coclass tree
\(\mathcal{T}^{2}\left(X_{3}^{2}\right)\).
A subroutine tests the transfer kernel type TKT of all descendants and selects
either the unique capable descendant with skeleton type \(\mathrm{c}.18\), respectively \(\mathrm{c}.21\),
or the unique descendant with simple hull type \(\mathrm{E}.6\), respectively \(\mathrm{E}.8\),
or the first or second descendant with simple hull type \(\mathrm{E}.14\), respectively \(\mathrm{E}.9\).
The selected non-mainline vertex is always checked for isomorphism
to the metabelianization of the appropriate quotient \(\mathcal{Q}^{(e,\ell)}_c\).
However, the bound \(n\le 8\),
is only due to the unavoidable limitations of the computational algebra system Magma
\cite{MAGMA2026},
and does not restrict the generality of the infinite cover limit \(\mathcal{C}^{(e)}\)
and its quotients \(\mathcal{Q}^{(e,\ell)}_c\).
\end{proof}


\noindent
A pure bottom up approach without top down constructions,
instead of using Algorithm 3.2 in
\cite[pp. 97--98]{Ma2018},
is able to reach coclass \(r=32\), nilpotency class \(c=63\), and logarithmic order \(r+c=95\),
without surpassing internal limits of Magma,
and thus provides even stronger support for Theorem
\ref{thm:ExplicitCovers},
from the experimental point of view.

Currently, a drawback of the quotients \(\mathcal{Q}^{(0,+1)}_c\) with \(e=0\), \(\ell=+1\), \(c\ge 5\),
which lead into a completely different realm,
namely the complicated brushwood of the complex transfer kernel type \(\mathrm{H}.4\),
is that they cannot be interpretated as covers of metabelianizations \(M_n^i\).


A common feature of all Theorems
\ref{thm:SimpleStageCriterion}--\ref{thm:ComplexPathStructure}
together are certain general requirements for Galois groups
\(\mathrm{Gal}(\mathrm{F}_p^q(k)/k)\), with a positive integer \(q\),
of the stages of the \(p\)-class field tower with an odd prime \(p\),
of quadratic number fields \(k=\mathbb{Q}(\sqrt{d})\).

Generally, any Galois group
\(G=\mathrm{Gal}(\mathrm{F}_3^q(k)/k)\), with a positive integer \(q\),
must be a \(\sigma\)-group,
possessing an automorphism which acts as inversion
on both cohomology groups \(H^i(G,\mathbb{F}_3)\), with \(i=1,2\), and the 
\(3\)-class field tower group \(S=\mathrm{Gal}(\mathrm{F}_3^\infty(k)/k)\)
additionally must satisfy
the Shafarevich bound for the relation rank,
\(2\le d_2(S)\le d_{2,\mathrm{max}}\)
with \(d_{2,\mathrm{max}}=2\) for imaginary \(k\), enforcing a Schur \(\sigma\)-group,
and \(d_{2,\mathrm{max}}=3\) for real \(k\), admitting also a Schur\(+1\) \(\sigma\)-group.
\cite{Sh1964,Ma2015c}.

For each simple type in section E
\cite[p. 36]{SoTa1934},
and separately for each complex type in the sections G and H
\cite[pp. 36--38]{SoTa1934},
the candidates for the metabelianization \(M=S/S^{\prime\prime}=\mathrm{Gal}(\mathrm{F}_3^2(k)/k)\)
and for the \(3\)-class field tower group \(S=\mathrm{Gal}(\mathrm{F}_3^\infty(k)/k)\) itself
are identified by a search along the descendant tree of
either the root \(Q=\langle 729,49\rangle\) or \(U=\langle 729,54\rangle\)
\cite{AHL1977,BEO2005},
constructed with the \(p\)-group generation algorithm by Newman
\cite{Nm1975,Nm1989}
and O'Brien
\cite{Ob1990,HEO2005}.
The claims for the ground state \(n=0\) and the lower excited states \(n=1,2,3,4\) 
are covered by initial computations with Magma
\cite{BCP1997,BCFS2025,MAGMA2026,MAGMA6561},
and evidence of statements for higher excited states \(n\ge 5\) is provided
for the metabelianizations \(M_n^i\) either by Corollary
\ref{cor:Metabelianization}
or by the \textit{periodicity of tree branches}, which was proven independently in
\cite{dS2001,EkLG2008,ELNO2013,LG1994,Sv1994}.


In the notation with relative identifiers
\cite{GNO2006},
and starting from one of the forks with bifurcation
\(F\in\lbrace Q,U\rbrace\) as roots,
for simple types,
the metabelianizations turn out to be of the shape
\(M=F(-\#1;1-\#1;1)^n\lbrack-\#1;i\rbrack\),
and
\(M=F(-\#1;1-\#1;1)^n\lbrack-\#1;i-\#1;1\rbrack\)
for complex types,
since, in the latter situation, for \(G=\pi(M)=F(-\#1;1-\#1;1)^n\lbrack-\#1;i\rbrack\),
an automorphism acts as inversion
on \(H^1(G,\mathbb{F}_3)\) only but not on \(H^2(G,\mathbb{F}_3)\).
So the group \(M\) is still a vertex of the coclass tree
\(\mathcal{T}^2(F)\) with exclusive step size \(s=1\)
and thus with fixed coclass \(2\).
It has depth \(\mathrm{dp}=1\) for simple types and
\(\mathrm{dp}=2\) for complex types.

However, periodic bifurcations must actually be exploited for
\(S=F(-\#2;1-\#1;1)^n\lbrack-\#2;i\rbrack\),
if the type is simple, and for
\(S=F(-\#2;1-\#1;1)^n\lbrack-\#2;i-\#1;1\rbrack\lbrack-\#2;1-\#1;1\rbrack^{t}\lbrack-\#2;2\rbrack\),
when the type is complex,
due to infinite branches with infinitely many candidates for each state \(n\ge 0\),
expressed by the second parameter \(t\ge 0\).


\subsection{Proofs for simple types}
\label{ss:SimpleProofs}

\begin{proof}
We prove Theorems
\ref{thm:SimpleStageCriterion}
and
\ref{thm:SimplePathStructure}
together.
The group theoretic invariants of the metabelianizations \(M_n^i\)
and the Schur \(\sigma\)-groups \(S_n^i\)
for any state \(n\ge 0\) and counters \(2\le i\le 4\) are: \\
logarithmic order,
\(\mathrm{lo}(M_n^i)=\mathrm{lo}(F)+2n+1=2n+7\),
\(\mathrm{lo}(S_n^i)=\mathrm{lo}(F)+3n+2=3n+8\), \\
nilpotency class,
\(\mathrm{cl}(M_n^i)=\mathrm{cl}(F)+2n+1=2n+5\),
\(\mathrm{cl}(S_n^i)=\mathrm{cl}(F)+2n+1=2n+5\), and \\
coclass,
\(\mathrm{cc}(M_n^i)=\mathrm{cc}(F)+0=2\),
\(\mathrm{cc}(S_n^i)=\mathrm{cc}(F)+n+1=n+3\), \\
since the fork possesses the invariants \(\mathrm{lo}(F)=6\), \(\mathrm{cl}(F)=4\), \(\mathrm{cc}(F)=2\). 

The invariants of the Schur\(+1\) \(\sigma\)-groups \(T_{n,u}^i\)
for excited states \(n\ge 1\), \(1\le u\le n\) and counters \(2\le i\le 4\) are: \\
logarithmic order,
\(\mathrm{lo}(T_{n,u}^i)=\mathrm{lo}(F)+3u+2(n-u)+1=2n+2u+7\), \\
nilpotency class,
\(\mathrm{cl}(T_{n,u}^i)=\mathrm{cl}(F)+2u+2(n-u)+1=2n+5\), and \\
coclass,
\(\mathrm{cc}(T_{n,u}^i)=\mathrm{cc}(F)+u=u+2\). 

In
\cite[Dfn. 3.3, Eqn. 6--7, p. 97]{Ma2018},
an infinite \textit{cover limit} group \(\mathcal{C}^{(e)}\),
due to M. F. Newman, is defined for two parameter values \(e=0,1\),
together with class-\(c\) quotients \(\mathcal{Q}^{(e,k)}_c\) of \(\mathcal{C}^{(e)}\)
with parameter \(k=-1,0,+1\) and nilpotency class \(c\ge 4\).
In
\cite[Thm. 3.5, Eqn. 8--12, pp. 98--100]{Ma2018},
it is shown that, for odd nilpotency class \(c=2n+5\ge 5\),
\(\mathcal{Q}^{(0,0)}_c\simeq S_n^2\), \(\mathcal{Q}^{(0,-1)}_c\simeq S_n^4\)
on the \(Q\)-tree, and
\(\mathcal{Q}^{(1,0)}_c\simeq S_n^2\), \(\mathcal{Q}^{(1,-1)}_c\simeq S_n^3\), \(\mathcal{Q}^{(1,+1)}_c\simeq S_n^4\)
on the \(U\)-tree.
Note that in \cite{Ma2018}, the Schur \(\sigma\)-group \(S_n^i\) is denoted by \(S_{2n+5,i-1}^{(n+3)}\) for \(2\le i\le 4\),
emphasizing the nilpotency class \(c=2n+5\) and the coclass \(r=n+3\), for each state \(n\ge 0\).
The results are translated to the present notation in Theorem
\ref{thm:ExplicitCovers}
(with \(k\) replaced by \(\ell\)).

Theorem
\ref{thm:SimplePathStructure}
is now an immediate consequence of the
parametrized polycyclic pc-presentations in Proposition
\ref{prp:TopDownQ}--\ref{prp:TopDownU}
and the relative identifiers in Corollary
\ref{cor:Quadratic},
which give the shape of \(M_n^i\),
the quotients of the mainline-\(r\) limit groups \(\mathcal{L}_{\mp 1}^{r}\) in Theorem
\ref{thm:MainLines},
which give the complete skeleton, and
the quotients of the cover limit group \(\mathcal{C}^{(e)}\) in Theorem
\ref{thm:ExplicitCovers},
which give the covers \(S_n^i\).

Theorem
\ref{thm:SimpleStageCriterion}
follows by computation of the abelian type invariants of second order, \(\alpha^{(2)}(k)\),
of all groups \(M_n^i\), \(S_n^i\), \(T_{n,u}^i\).
\end{proof}


\subsection{Proofs for complex types}
\label{ss:ComplexProofs}

\begin{proof}
We prove Theorems
\ref{thm:ComplexStageCriterion}
and
\ref{thm:ComplexPathStructure}
together.
So far we have only paid attention to the Schur \(\sigma\)-groups
among the candidates for \(S\).
For \textit{imaginary} \(k\), these are the only non-metabelian candidates.
For \textit{real} \(k\) with \textit{simple} types,
the non-metabelian Schur\(+1\) \(\sigma\)-groups \(T_{n,u}^i\)
arise as additional candidates for \(S\).
For \textit{real} \(k\) with \textit{complex} types, however,
various further non-metabelian Schur\(+1\) \(\sigma\)-groups
are also admissible.
They all share a common metabelianization \(M=S/S^{\prime\prime}\).

At this point, the (logarithmic) abelian type invariants of the second order
\(\alpha^{(2)}(k)\) come into the play. 
For \textit{simple} types,
the metabelianization \(M\) is a terminal vertex of its descendant tree, and
we obtain \textit{sharp necessary and sufficient criteria}, since
the \textit{regular} components of the \textit{stabilization} are
tame, \(\lbrack 21;\alpha_0,\mathbf{(21)^3}\rbrack\), in the case \(S=M\) of a tower with two stages, and
wild, \(\lbrack 21;\alpha_0,\mathbf{(31)^3}\rbrack\), when the tower has precisely three stages.
On the tree with root \(Q\),
the \textit{singular} component of the \textit{stabilization} admits an additional decision, it is
tame, \(\lbrack 1^3;\alpha_0,\mathbf{(1^3)^3},(1^2)^9\rbrack\), when \(\ell_3(k)=2\),
i.e. \(S=M\) is a metabelian Schur\(+1\) \(\sigma\)-group, and
wild, \(\lbrack 1^3;\alpha_0,\mathbf{(21^2)^3},(1^2)^9\rbrack\), when \(\ell_3(k)=3\),
i.e. \(S\ne M\) is a non-metabelian Schur \(\sigma\)-group \(S_n^i\) or Schur\(+1\) \(\sigma\)-group \(T_{n,u}^i\).
Here the \textbf{hetero}cyclic \textit{polarization} \(\lbrack(n+3,n+2);\alpha_0,(n+\mathbf{3},n+1,1)^3\rbrack\)
is always tame and does not enable a decision.
See \cite[\S\ 4, Thm. 2]{AoMa2025} for this terminology concerning the first order ATI.
The Schur \(\sigma\)-groups \(S\) are drawn
in Figure \ref{fig:TreeStrucSimple} for \textit{simple} types, and
in Figure \ref{fig:TreeStrucComplex} for \textit{complex} types.
The path between \(S\) and \(M\) always represents a so-called \textit{fork topology}
with respect to the first bifurcation \(F\) in the skeleton.

For \textit{real} \(k\) and \textit{complex} types,
there arise additional possibilities of Schur\(+1\) \(\sigma\)-groups,
besides the Schur \(\sigma\)-groups described for imaginary \(k\).
These additional candidates for complex types
are not drawn in Figure \ref{fig:TreeStrucComplex}
in order to avoid too complicated tree structures.
For the GS, they are drawn in Figure
\ref{fig:ComplexTopologies}.

Firstly, a so-called \textit{child topology} for the path between \(S\) and \(M\)
can be realized by one of the immediate non-metabelian descendants
\(R_{0j}^i=M-\#1;j\), \(j=1,2,3\) for type H.4, \(j=1,2\) for type G.16,
of the non-terminal metabelianization
\(M=M_n^i=F(-\#1;1-\#1;1)^n\lbrack-\#1;i-\#1;1\rbrack\).
They share a common tame stabilization.
The polarization \(\lbrack(n+3,n+2);\alpha_0,(n+\mathbf{4},n+1,1)^3\rbrack\) is wild
when \(j=1,2\) for type H.4 and \(j=1\) for type G.16.
These are the only cases where the tower even has \textit{precisely three stages}.
The crux concerning two stage towers is that \(j=3\) for type H.4 and \(j=2\) for type G.16
shows a tame polarization \(\lbrack(n+3,n+2);\alpha_0,(n+\mathbf{3},n+1,1)^3\rbrack\)
and thus cannot be distinguished from a metabelian tower.

Secondly, further \textit{fork topologies} for the path between \(S\) and \(M\)
can be due to immediate descendants
\(G\lbrack-\#1;j\rbrack\)
with step size \(s=1\) of
\(G=\pi(S_{n,t}^i)=F(-\#2;1-\#1;1)^n\lbrack-\#2;i-\#1;1\rbrack\lbrack-\#2;1-\#1;1\rbrack^{t}\)
with \(i=5,6\).
They share a common wild stabilization.
The polarization \(\lbrack(n+3,n+2);\alpha_0,(n+\mathbf{4},n+1,1)^3\rbrack\) is wild
when \(j=3\) for type H.4 and \(j=1\) for type G.16.
The polarization \(\lbrack(n+3,n+2);\alpha_0,(n+\mathbf{3},n+1,1)^3\rbrack\) is tame
when \(j=1,4\) for type H.4 and \(j=3,4\) for type G.16.
For both types, \(j=2\) gives rise to a finite sapling
with depth \(2n\)
sharing a common wild polarization.
See Figure
\ref{fig:ComplexTopologies},
for the ground state GS \(n=0\).

In Theorem
\ref{thm:ComplexPathStructure},
the modified shape of \(M_n^i=F(-\#1;1-\#1;1)^n\lbrack-\#1;i-\#1;1\rbrack\)
is even proven for unbounded states \(n\ge 0\),
due to the parametrized polycyclic pc-presentations in Proposition
\ref{prp:TopDownQ}--\ref{prp:TopDownU}
and an easy generalization of the relative identifiers in Corollary
\ref{cor:Quadratic}.
The shape of the covers \(S_{n,t}^i\) is investigated experimentally
for bounded states \(0\le n\le 4\) only.

Theorem
\ref{thm:ComplexStageCriterion}
follows by computation of the abelian type invariants of second order, \(\alpha^{(2)}(k)\),
of all involved groups.
\end{proof}


\section{Recall of simple types in the imaginary quadratic scenario, \(d<0\)}
\label{s:RecallImaginary}

\noindent
Suppose \(p\) is an \textit{odd} prime number.
Let \(k=\mathbb{Q}(\sqrt{d})\), with fundamental discriminant \(d<0\), be an \textit{imaginary} quadratic number field
possessing an elementary bicyclic \(p\)-class group \(\mathrm{Cl}_p(k)=(\mathbb{Z}/p\mathbb{Z})^2\).
Then the Galois group \(S=\mathrm{Gal}(\mathrm{F}_p^\infty(k)/k)\)
of the maximal unramified pro-\(p\)-extension of \(k\)
is a non-abelian Schur \(\sigma\)-group
\cite{Ag1998}
with balanced presentation and a generator and relator inverting automorphism \(\sigma\) of order \(2\)
\cite{KoVe1975}.
The second derived quotient \(S/S^{\prime\prime}\) of \(S\) is isomorphic to the Galois group
\(M=\mathrm{Gal}(\mathrm{F}_p^2(k)/k)\)
of the maximal metabelian unramified \(p\)-extension of \(k\).


\noindent
For \(p=3\), we recall the well-known facts concerning the
ground state (GS) and the excited states (ES) of \(3\)-tower groups \(S\) of \(k\)
with capitulation type \(\varkappa(k)\) in section E of Scholz and Taussky
\cite{SoTa1934}:

\begin{theorem}
\label{thm:ThreeTowerGround}
Let the (logarithmic) abelian type invariants (ATI) of
the \(4\) unramified cyclic cubic extensions \(E_1,\ldots,E_4\) of \(k\) be assigned by
either \(\alpha(k)=\lbrack 32,111,21,21\rbrack\) or \(\alpha(k)=\lbrack 32,21,21,21\rbrack\).
In this context, we speak about the \textbf{ground state} of the groups \(M\) and \(S\)
with \textbf{hetero}cyclic polarization \((32)\),
characterized by the nilpotency class \(\mathrm{cl}(M)=3+2=5\) of \(M\).
In terms of the identifiers \(\langle\mathrm{ord},\mathrm{id}\rangle\) in the SmallGroups database
\cite{BEO2005}
of the computer algebra system Magma
\cite{MAGMA2026,MAGMA6561},
the groups \(M\) and \(S\) are given as follows in dependence of the capitulation type:
\begin{enumerate}
\item
if \(\varkappa(k)\sim (1122)\), type E.6,
then \(M=\langle 2187,288\rangle\) and \(S=\langle 6561,616\rangle\);
\item
if \(\varkappa(k)\sim (3122)\), type E.14,
then either \(M=\langle 2187,289\rangle\) and \(S=\langle 6561,617\rangle\) \\
or \(M=\langle 2187,290\rangle\) and \(S=\langle 6561,618\rangle\);
\item
if \(\varkappa(k)\sim (1231)\), type E.8,
then \(M=\langle 2187,304\rangle\) and \(S=\langle 6561,622\rangle\);
\item
if \(\varkappa(k)\sim (2231)\), type E.9,
then either \(M=\langle 2187,302\rangle\) and \(S=\langle 6561,620\rangle\) \\
or \(M=\langle 2187,306\rangle\) and \(S=\langle 6561,624\rangle\).
\end{enumerate}
The path between \(M\) and \(S\) is drawn in Figure
\ref{fig:TreeStrucSimple}, starting at the symbol GS.
For the types E.6 and E.14, the root is \(R=\langle 243,6\rangle\),
and for the types E.8 and E.9, the root is \(R=\langle 243,8\rangle\).
\end{theorem}

\begin{proof}
For \(\varkappa(k)\sim (2231)\), type E.9, the statement was proved in
\cite[Cor. 4.1.1, p. 775]{BuMa2015}.
For the other types, the proof was given in
\cite{Ma2015a}
and
\cite{Ma2015b}.
The trees with minimal discriminants \(d\) of prototypes are drawn in
\cite[Fig. 1--2, pp. 24--25]{Ma2015c}.
\end{proof}


\begin{remark}
\label{rmk:FixedPoints}
The four capitulation types in section E
can be distinguished by the number of fixed points (FP):
type E.14, \(\varkappa(k)\sim (3122)\), has \(0\) FP,
type E.6, \(\varkappa(k)\sim (1122)\), has \(1\) FP,
type E.9, \(\varkappa(k)\sim (2231)\), has \(2\) FP, and
type E.8, \(\varkappa(k)\sim (1231)\), has \(3\) FP.
They are \textit{simple} types.
\end{remark}


\begin{theorem}
\label{thm:ThreeTowerExcited}
With an integer \(n\ge 1\),
let the (logarithmic) abelian type invariants (ATI) of
the \(4\) unramified cyclic cubic extensions \(E_1,\ldots,E_4\) of \(k\) be fixed by
either \(\alpha(k)=\lbrack (n+3,n+2),111,21,21\rbrack\) or \(\alpha(k)=\lbrack (n+3,n+2),21,21,21\rbrack\).
In this context, we speak about the \textbf{\(n\)-th excited state} of the groups \(M\) and \(S\),
characterized by the nilpotency class \(\mathrm{cl}(M)=n+3+n+2=2n+5\) of \(M\).
Then the groups \(M\) and \(S\)
with \textbf{hetero}cyclic polarization \((n+3,n+2)\),
are given by the following paths:
\begin{equation}
\label{eqn:Structure}
\begin{aligned}
M &= R(-\#1;1)(-\#1;1-\#1;1)^n\lbrack-\#1;i\rbrack \text{ and } \\
S &= R(-\#1;1)(-\#2;1-\#1;1)^n\lbrack-\#2;i\rbrack,
\end{aligned}
\end{equation}
in terms of relative ANUPQ-identifiers
\cite{GNO2006}.
The dependency of the capitulation type is expressed by \(2\le i\le 4\).
For \(1\le n\le 4\), the path between \(M\) and \(S\) is drawn in Figure
\ref{fig:TreeStrucSimple}, starting at the symbol ES \(n\).
For the types E.6 and E.14, the root is \(R=\langle 243,6\rangle\),
and for the types E.8 and E.9, the root is \(R=\langle 243,8\rangle\).
\end{theorem}

\begin{proof}
The proof was given in
\cite{Ma2015a}
and
\cite{Ma2015b}.
Note that \(R(-\#1;1)=\langle 729,49\rangle\) for \(R=\langle 243,6\rangle\), and
\(R(-\#1;1)=\langle 729,54\rangle\) for \(R=\langle 243,8\rangle\),
is the first tree vertex with bifurcation (\(1\le s\le 2\)).
\end{proof}


\begin{example}
\label{exm:TreeStrucTri}
For the type E.9, \(\varkappa(k)\sim (2231)\),
the minimal discriminants of prototypes for the states are given by
\(d=-9\,748\) for GS,
\(d=-297\,079\) for ES 1,
\(d=-1\,088\,808\) for ES 2,
\(d=-11\,091\,140\) for ES 3, and
\(d=-94\,880\,548\) for ES 4,
computed with Magma
\cite{MAGMA2026}
using 
\cite{HEO2005}.
\end{example}


\section{Tree diagrams}
\label{s:Diagrams}

\noindent
In both tree diagrams of Figure
\ref{fig:TreeStrucSimple}
and
\ref{fig:TreeStrucComplex},
there are two possible realizations of the root \(R\)
on the top of the tree diagram.
Both of these finite \(3\)-groups
have nilpotency class \(\mathrm{cl}(R)=3\) and coclass \(\mathrm{cc}(R)=2\).
In both cases, \(R\) belongs to Hall's isoclinism family \(\Phi_6\)
\cite[p. 139]{Ha1940}.
On The \(Q\)-tree, the root is \(R=\langle 243,6\rangle\)
with skeleton type c.18, \(\varkappa(k)\sim(0122)\).
On The \(U\)-tree, the root is \(R=\langle 243,8\rangle\)
with skeleton type c.21, \(\varkappa(k)\sim(0231)\).

\subsection{The subtree \({}^\ast\mathcal{T}(R)\) pruned from complex types}
\label{ss:SimpleTree}


\noindent
Remarks concerning Figure
\ref{fig:TreeStrucSimple}:
The figure is restricted to the root region
up to logarithmic order \(\mathrm{lo}=20\),
and to the \textit{pruned subtree} \({}^\ast\mathcal{T}(R)<\mathcal{T}(R)\)
of the complete descendant tree of the root \(R\),
where the complex types are eliminated,
and only simple types in section E of Scholz and Taussky
\cite{SoTa1934}
remain as terminal leaves,
connected by mainline vertices of skeleton type.
By \({}^\ast\mathcal{T}^2(R)\), we denote the pruned subtree of \({}^\ast\mathcal{T}(R)\)
with fixed coclass \(\mathrm{cc}=2\), a so-called \textit{coclass tree}.
In terms of relative ANUPQ-identifiers
\cite{GNO2006},
the Schur \(\sigma\)-groups
\cite{BBH2017}
with simple types are given by periodic bifurcations as in Formula
\eqref{eqn:SimplePath},
\[
S=S_n^i:=R(-\#1;1)(-\#2;1-\#1;1)^n\lbrack -\#2;i\rbrack,\ 2\le i\le 4,
\]
a unique Schur \(\sigma\)-group for each state \(n\ge 0\).
Parentheses enclose vertices with skeleton type,
and brackets surround vertices with simple types.
The situation is unproblematic with respect to ambiguities,
since the Schur \(\sigma\)-group \(S\) with \(M=S/S^{\prime\prime}\)
is unique, but Schur\(+1\) \(\sigma\)-groups
\cite{BBH2021}
with simple types also exist.
The metabelianizations are given as second derived quotients as in Formula
\eqref{eqn:SimplePath},
\[
M=M_n^i:=S/S^{\prime\prime}=R(-\#1;1)(-\#1;1-\#1;1)^n\lbrack -\#1;i\rbrack,\ 2\le i\le 4,
\]
on the pruned coclass subtree \({}^\ast\mathcal{T}^2(R)\).


{\tiny


\begin{figure}[ht]
\caption{Schur \(\sigma\)-groups with \textbf{simple} type in excited states on \({}^\ast\mathcal{T}(R)\)}
\label{fig:TreeStrucSimple}

\setlength{\unitlength}{0.8cm}
\begin{picture}(18,17)(-8,-16)


\put(3,-0.5){\makebox(0,0)[lt]{Legend:}}
\put(3,-1.5){\circle{0.3}}
\put(3,-1.5){\circle*{0.1}}
\put(3.5,-1.5){\makebox(0,0)[lc]{\(\ldots\) fork \(F\) with bifurcation}}
\put(3,-2){\circle*{0.1}}
\put(3.5,-2){\makebox(0,0)[lc]{\(\ldots\) metabelian skeleton type group}}
\put(3,-2.5){\circle{0.2}}
\put(3.5,-2.5){\makebox(0,0)[lc]{\(\ldots\) metabelianization \(M=S/S^{\prime\prime}\)}}
\put(2.9,-3.1){\framebox(0.2,0.2){}}
\put(3.5,-3){\makebox(0,0)[lc]{\(\ldots\) Schur \(\sigma\)-group \(S\)}}
\put(3,-3.5){\makebox(0,0)[lc]{\(\nwarrow\)}}
\put(3.5,-3.5){\makebox(0,0)[lc]{\(\ldots\) projection \(S\to S/S^{\prime\prime}\)}}
\put(2.95,-4.05){\framebox(0.1,0.1){}}
\put(3.5,-4){\makebox(0,0)[lc]{\(\ldots\) Schur\(+1\) \(\sigma\)-group \(T\)}}
\put(3,-4.5){\makebox(0,0)[lc]{GS}}
\put(3.5,-4.5){\makebox(0,0)[lc]{\(\ldots\) ground state}}
\put(3,-5){\makebox(0,0)[lc]{ES}}
\put(3.5,-5){\makebox(0,0)[lc]{\(\ldots\) excited state}}


\put(-5,0.5){\makebox(0,0)[cb]{Order}}
\put(-5,0){\line(0,-1){15}}
\multiput(-5.1,0)(0,-1){16}{\line(1,0){0.2}}
\put(-5.2,0){\makebox(0,0)[rc]{\(243\)}}
\put(-4.8,0){\makebox(0,0)[lc]{\(3^5\)}}
\put(-5.2,-1){\makebox(0,0)[rc]{\(729\)}}
\put(-4.8,-1){\makebox(0,0)[lc]{\(3^6\)}}
\put(-5.2,-2){\makebox(0,0)[rc]{\(2\,187\)}}
\put(-4.8,-2){\makebox(0,0)[lc]{\(3^7\)}}
\put(-5.2,-3){\makebox(0,0)[rc]{\(6\,561\)}}
\put(-4.8,-3){\makebox(0,0)[lc]{\(3^8\)}}
\put(-5.2,-4){\makebox(0,0)[rc]{\(19\,683\)}}
\put(-4.8,-4){\makebox(0,0)[lc]{\(3^9\)}}
\put(-5.2,-5){\makebox(0,0)[rc]{\(59\,049\)}}
\put(-4.8,-5){\makebox(0,0)[lc]{\(3^{10}\)}}
\put(-5.2,-6){\makebox(0,0)[rc]{\(177\,147\)}}
\put(-4.8,-6){\makebox(0,0)[lc]{\(3^{11}\)}}
\put(-5.2,-7){\makebox(0,0)[rc]{\(531\,441\)}}
\put(-4.8,-7){\makebox(0,0)[lc]{\(3^{12}\)}}
\put(-5.2,-8){\makebox(0,0)[rc]{\(1\,594\,323\)}}
\put(-4.8,-8){\makebox(0,0)[lc]{\(3^{13}\)}}
\put(-5.2,-9){\makebox(0,0)[rc]{\(4\,782\,969\)}}
\put(-4.8,-9){\makebox(0,0)[lc]{\(3^{14}\)}}
\put(-5.2,-10){\makebox(0,0)[rc]{\(14\,348\,907\)}}
\put(-4.8,-10){\makebox(0,0)[lc]{\(3^{15}\)}}
\put(-5.2,-11){\makebox(0,0)[rc]{\(43\,046\,721\)}}
\put(-4.8,-11){\makebox(0,0)[lc]{\(3^{16}\)}}
\put(-5.2,-12){\makebox(0,0)[rc]{\(129\,140\,163\)}}
\put(-4.8,-12){\makebox(0,0)[lc]{\(3^{17}\)}}
\put(-5.2,-13){\makebox(0,0)[rc]{\(387\,420\,489\)}}
\put(-4.8,-13){\makebox(0,0)[lc]{\(3^{18}\)}}
\put(-5.2,-14){\makebox(0,0)[rc]{\(1\,162\,261\,467\)}}
\put(-4.8,-14){\makebox(0,0)[lc]{\(3^{19}\)}}
\put(-5.2,-15){\makebox(0,0)[rc]{\(3\,486\,784\,401\)}}
\put(-4.8,-15){\makebox(0,0)[lc]{\(3^{20}\)}}
\put(-5,-15){\vector(0,-1){1}}

\put(-2,0.2){\makebox(0,0)[cb]{Root \(R=X_3^2\)}}
\put(-1.7,-1){\makebox(0,0)[lc]{Fork \(F=X_4^2\)}}
\put(-2,-1){\circle{0.3}}
\put(-2,-1){\line(-1,-1){1}}
\put(-3.3,-2){\makebox(0,0)[rc]{GS}}
\put(-3,-2.3){\makebox(0,0)[ct]{\(M_0^i\)}}
\put(-3,-2){\circle{0.2}}

\multiput(-2,0)(0,-1){11}{\circle*{0.1}}

\put(-2,-1){\line(0,-1){2}}
\put(-1.8,-3){\makebox(0,0)[lc]{\(X_6^2\)}}
\put(-2,-3){\line(-1,-1){1}}
\put(-3.3,-4){\makebox(0,0)[rc]{ES \(1\)}}
\put(-3,-4.3){\makebox(0,0)[ct]{\(M_1^i\)}}
\put(-3,-4){\circle{0.2}}

\put(-2,-3){\line(0,-1){2}}
\put(-1.8,-5){\makebox(0,0)[lc]{\(X_8^2\)}}
\put(-2,-5){\line(-1,-1){1}}
\put(-3.3,-6){\makebox(0,0)[rc]{ES \(2\)}}
\put(-3,-6.3){\makebox(0,0)[ct]{\(M_2^i\)}}
\put(-3,-6){\circle{0.2}}

\put(-2,-5){\line(0,-1){2}}
\put(-1.8,-7){\makebox(0,0)[lc]{\(X_{10}^2\)}}
\put(-2,-7){\line(-1,-1){1}}
\put(-3.3,-8){\makebox(0,0)[rc]{ES \(3\)}}
\put(-3,-8.3){\makebox(0,0)[ct]{\(M_3^i\)}}
\put(-3,-8){\circle{0.2}}

\put(-2,-7){\line(0,-1){2}}
\put(-1.8,-9){\makebox(0,0)[lc]{\(X_{12}^2\)}}
\put(-2,-9){\line(-1,-1){1}}
\put(-3.3,-10){\makebox(0,0)[rc]{ES \(4\)}}
\put(-3,-10.3){\makebox(0,0)[ct]{\(M_4^i\)}}
\put(-3,-10){\circle{0.2}}

\put(-2,-9){\vector(0,-1){3}}
\put(-2,-12.3){\makebox(0,0)[cc]{\({}^\ast\mathcal{T}^2(R)\)}}
\put(-2,-12.7){\makebox(0,0)[cc]{Infinite}}
\put(-2,-13.1){\makebox(0,0)[cc]{mainline}}

\put(-2,0){\line(0,-1){1}}
\put(-2,-1){\line(1,-1){2}}
\put(0.2,-3){\makebox(0,0)[lc]{\(X_5^3\)}}
\put(0.2,-4){\makebox(0,0)[lc]{\(X_6^3\)}}
\put(-2,-1){\line(1,-2){1}}
\put(-1.1,-3.1){\framebox(0.2,0.2){}}
\put(-1,-3.4){\makebox(0,0)[ct]{\(S_0^i\)}}

\put(0,-3){\line(0,-1){7}}
\multiput(0,-4)(0,-2){4}{\line(-1,-1){1}}
\multiput(-1.05,-5.05)(0,-2){4}{\framebox(0.1,0.1){}}
\put(-1,-5.4){\makebox(0,0)[ct]{\(T_{1,0}^i\)}}
\put(-1,-7.4){\makebox(0,0)[ct]{\(T_{1,1}^i\)}}
\put(-1,-9.4){\makebox(0,0)[ct]{\(T_{1,2}^i\)}}
\put(-1,-11.4){\makebox(0,0)[ct]{\(T_{1,3}^i\)}}
\put(0,-4){\line(1,-1){2}}
\put(2.2,-6){\makebox(0,0)[lc]{\(X_7^4\)}}
\put(2.2,-7){\makebox(0,0)[lc]{\(X_8^4\)}}
\put(0,-4){\line(1,-2){1}}
\put(0.9,-6.1){\framebox(0.2,0.2){}}
\put(1,-6.4){\makebox(0,0)[ct]{\(S_1^i\)}}

\put(2,-6){\line(0,-1){5}}
\multiput(2,-7)(0,-2){3}{\line(-1,-1){1}}
\multiput(0.95,-8.05)(0,-2){3}{\framebox(0.1,0.1){}}
\put(1,-8.4){\makebox(0,0)[ct]{\(T_{2,0}^i\)}}
\put(1,-10.4){\makebox(0,0)[ct]{\(T_{2,1}^i\)}}
\put(1,-12.4){\makebox(0,0)[ct]{\(T_{2,2}^i\)}}
\put(2,-7){\line(1,-1){2}}
\put(4.2,-9){\makebox(0,0)[lc]{\(X_9^5\)}}
\put(4.2,-10){\makebox(0,0)[lc]{\(X_{10}^5\)}}
\put(2,-7){\line(1,-2){1}}
\put(2.9,-9.1){\framebox(0.2,0.2){}}
\put(3,-9.4){\makebox(0,0)[ct]{\(S_2^i\)}}

\put(4,-9){\line(0,-1){3}}
\multiput(4,-10)(0,-2){2}{\line(-1,-1){1}}
\multiput(2.95,-11.05)(0,-2){2}{\framebox(0.1,0.1){}}
\put(3,-11.4){\makebox(0,0)[ct]{\(T_{3,0}^i\)}}
\put(3,-13.4){\makebox(0,0)[ct]{\(T_{3,1}^i\)}}
\put(4,-10){\line(1,-1){2}}
\put(6.2,-12){\makebox(0,0)[lc]{\(X_{11}^6\)}}
\put(6.2,-13){\makebox(0,0)[lc]{\(X_{12}^6\)}}
\put(4,-10){\line(1,-2){1}}
\put(4.9,-12.1){\framebox(0.2,0.2){}}
\put(5,-12.4){\makebox(0,0)[ct]{\(S_3^i\)}}

\put(6,-12){\line(0,-1){1}}
\put(6,-13){\line(-1,-1){1}}
\put(4.95,-14.05){\framebox(0.1,0.1){}}
\put(5,-14.4){\makebox(0,0)[ct]{\(T_{4,0}^i\)}}
\put(6,-13){\vector(1,-1){2}}
\put(9,-15.5){\makebox(0,0)[cc]{Infinite}}
\put(9,-15.9){\makebox(0,0)[cc]{maintrunk}}
\put(6,-13){\line(1,-2){1}}
\put(6.9,-15.1){\framebox(0.2,0.2){}}
\put(7,-15.4){\makebox(0,0)[ct]{\(S_4^i\)}}

\put(-1,-3){\vector(-2,1){1.8}}
\put(1,-6){\vector(-2,1){3.8}}
\put(3,-9){\vector(-2,1){5.8}}
\put(5,-12){\vector(-2,1){7.8}}
\put(7,-15){\vector(-2,1){9.8}}

\end{picture}

\end{figure}
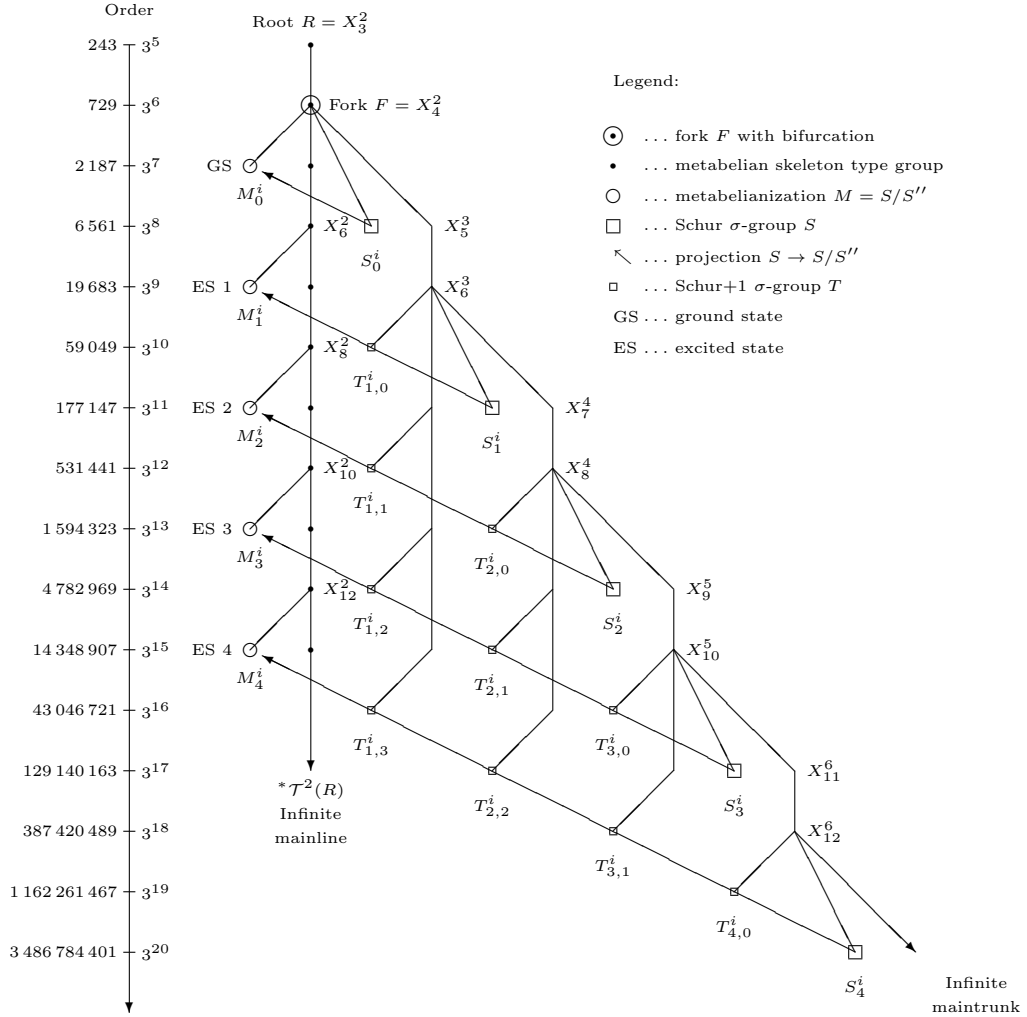

}


\noindent
The differences between the two possible realizations can be summarized as follows:

\noindent
\(\bullet\)
\textbf{The \(Q\)-tree:}
It is pruned from the complex type H.4, \(\varkappa(k)\sim(2122)\),
and the skeleton type of the mainline is c.18, \(\varkappa(k)\sim(0122)\).
In particular, the root \(R=\langle 243,6\rangle\) and its immediate descendant,
the fork \(Q=\langle 729,49\rangle=\langle 243,6\rangle(-\#1;1)\),
are of skeleton type.
The fork \(F=Q\) has a bifurcation, due to \textit{nuclear rank} \(\nu=2\).
The type of the terminal leaves is 
E.6, \(\varkappa(k)\sim(1122)\), for \(i=2\), and
E.14, \(\varkappa(k)\sim(3122)\), for \(i=3,4\).

\noindent
\(\bullet\)
\textbf{The \(U\)-tree:}
It is pruned from the complex type G.16, \(\varkappa(k)\sim(4231)\),
and the skeleton type of the mainline is c.21, \(\varkappa(k)\sim(0231)\).
In particular, the root \(R=\langle 243,8\rangle\) and its immediate descendant,
the fork \(U=\langle 729,54\rangle=\langle 243,8\rangle(-\#1;1)\),
possess the skeleton type.
The fork \(F=U\) has a bifurcation to different \textit{step sizes} \(s=1\) and \(s=2\).
The type of the terminal leaves is 
E.8, \(\varkappa(k)\sim(1231)\), for \(i=2\), and
E.9, \(\varkappa(k)\sim(2231)\), for \(i=3,4\).



\subsection{The subtree \({}_\ast\mathcal{T}(R)\) pruned from simple types}
\label{ss:ComplexTree}


\noindent
Remarks concerning Figure
\ref{fig:TreeStrucComplex}:
The figure is restricted to the root region
up to logarithmic order \(\mathrm{lo}=20\),
and to the \textit{pruned subtree} \({}_\ast\mathcal{T}(R)<\mathcal{T}(R)\)
of the complete descendant tree of the root \(R\),
where the simple types in section E of Scholz and Taussky 
\cite{SoTa1934}
are eliminated,
and only complex types
remain as finite or infinite branches,
connected by mainline vertices of skeleton type.
By \({}_\ast\mathcal{T}^2(R)\), we denote the pruned subtree of \({}_\ast\mathcal{T}(R)\)
with fixed coclass \(\mathrm{cc}=2\), a so-called \textit{coclass tree}.

The differences between the two possible realizations can be summarized as follows:

\noindent
\(\bullet\)
\textbf{The \(Q\)-tree:}
It is pruned from the simple types,
E.6, \(\varkappa(k)\sim(1122)\), and
E.14, \(\varkappa(k)\sim(3122)\).
The skeleton type of the mainline is c.18, \(\varkappa(k)\sim(0122)\).
In particular, the root \(R=\langle 243,6\rangle\) and its immediate descendant,
the fork \(Q=\langle 729,49\rangle=\langle 243,6\rangle(-\#1;1)\),
are of skeleton type.
The fork \(F=Q\) has a bifurcation, due to \textit{nuclear rank} \(\nu=2\).
The type of the infinite branches is the
complex hull type H.4, \(\varkappa(k)\sim(2122)\).

\noindent
\(\bullet\)
\textbf{The \(U\)-tree:}
It is pruned from the simple types,
E.8, \(\varkappa(k)\sim(1231)\), and
E.9, \(\varkappa(k)\sim(2231)\).
The skeleton type of the mainline is c.21, \(\varkappa(k)\sim(0231)\).
In particular, the root \(R=\langle 243,8\rangle\) and its immediate descendant,
the fork \(U=\langle 729,54\rangle=\langle 243,8\rangle(-\#1;1)\),
possess the skeleton type.
The fork \(F=U\) has a bifurcation to different \textit{step sizes} \(s=1\) and \(s=2\).
The type of the infinite branches is the
complex hull type G.16, \(\varkappa(k)\sim(4231)\)



{\tiny


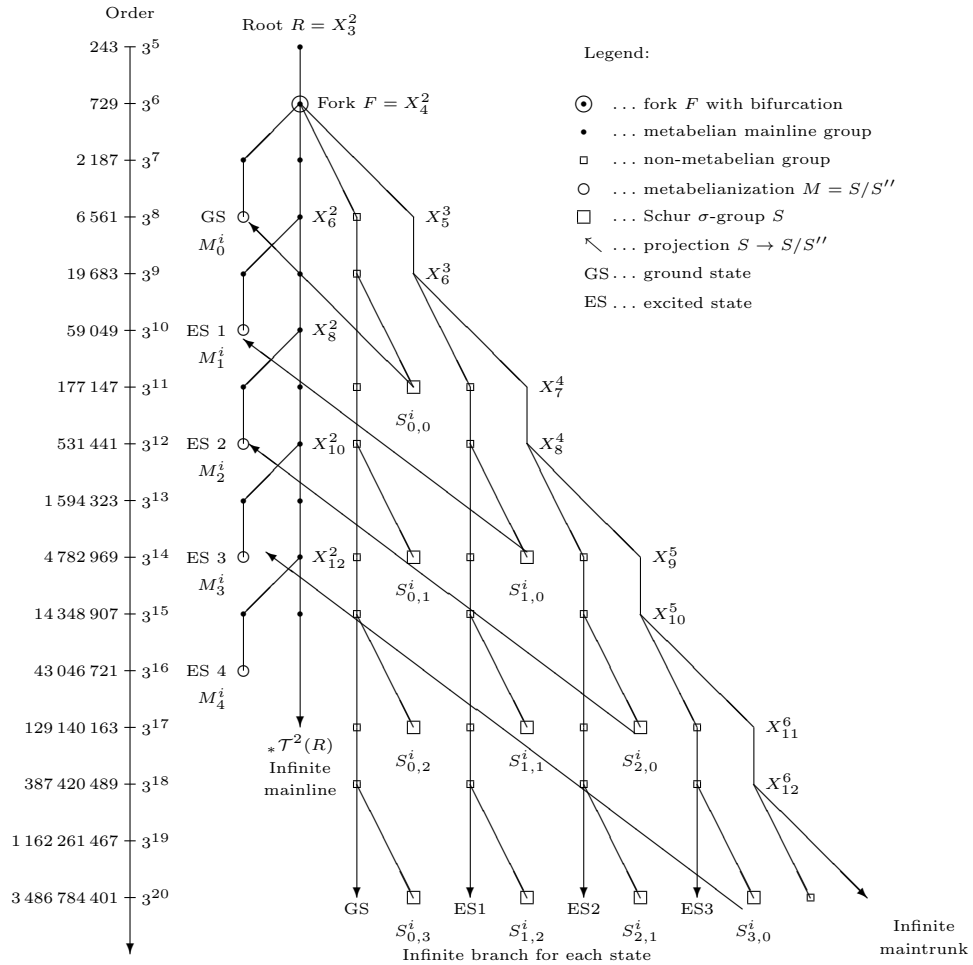
\begin{figure}[ht]
\caption{Schur \(\sigma\)-groups with \textbf{complex} type in excited states on \({}_\ast\mathcal{T}(R)\)}
\label{fig:TreeStrucComplex}

\setlength{\unitlength}{0.75cm}
\begin{picture}(18,17)(-8,-16)


\put(3,0){\makebox(0,0)[lt]{Legend:}}
\put(3,-1){\circle{0.3}}
\put(3,-1){\circle*{0.1}}
\put(3.5,-1){\makebox(0,0)[lc]{\(\ldots\) fork \(F\) with bifurcation}}
\put(3,-1.5){\circle*{0.1}}
\put(3.5,-1.5){\makebox(0,0)[lc]{\(\ldots\) metabelian mainline group}}
\put(2.95,-2.05){\framebox(0.1,0.1){}}
\put(3.5,-2){\makebox(0,0)[lc]{\(\ldots\) non-metabelian group}}
\put(3,-2.5){\circle{0.2}}
\put(3.5,-2.5){\makebox(0,0)[lc]{\(\ldots\) metabelianization \(M=S/S^{\prime\prime}\)}}
\put(2.9,-3.1){\framebox(0.2,0.2){}}
\put(3.5,-3){\makebox(0,0)[lc]{\(\ldots\) Schur \(\sigma\)-group \(S\)}}
\put(3,-3.5){\makebox(0,0)[lc]{\(\nwarrow\)}}
\put(3.5,-3.5){\makebox(0,0)[lc]{\(\ldots\) projection \(S\to S/S^{\prime\prime}\)}}
\put(3,-4){\makebox(0,0)[lc]{GS}}
\put(3.5,-4){\makebox(0,0)[lc]{\(\ldots\) ground state}}
\put(3,-4.5){\makebox(0,0)[lc]{ES}}
\put(3.5,-4.5){\makebox(0,0)[lc]{\(\ldots\) excited state}}


\put(-5,0.5){\makebox(0,0)[cb]{Order}}
\put(-5,0){\line(0,-1){15}}
\multiput(-5.1,0)(0,-1){16}{\line(1,0){0.2}}
\put(-5.2,0){\makebox(0,0)[rc]{\(243\)}}
\put(-4.8,0){\makebox(0,0)[lc]{\(3^5\)}}
\put(-5.2,-1){\makebox(0,0)[rc]{\(729\)}}
\put(-4.8,-1){\makebox(0,0)[lc]{\(3^6\)}}
\put(-5.2,-2){\makebox(0,0)[rc]{\(2\,187\)}}
\put(-4.8,-2){\makebox(0,0)[lc]{\(3^7\)}}
\put(-5.2,-3){\makebox(0,0)[rc]{\(6\,561\)}}
\put(-4.8,-3){\makebox(0,0)[lc]{\(3^8\)}}
\put(-5.2,-4){\makebox(0,0)[rc]{\(19\,683\)}}
\put(-4.8,-4){\makebox(0,0)[lc]{\(3^9\)}}
\put(-5.2,-5){\makebox(0,0)[rc]{\(59\,049\)}}
\put(-4.8,-5){\makebox(0,0)[lc]{\(3^{10}\)}}
\put(-5.2,-6){\makebox(0,0)[rc]{\(177\,147\)}}
\put(-4.8,-6){\makebox(0,0)[lc]{\(3^{11}\)}}
\put(-5.2,-7){\makebox(0,0)[rc]{\(531\,441\)}}
\put(-4.8,-7){\makebox(0,0)[lc]{\(3^{12}\)}}
\put(-5.2,-8){\makebox(0,0)[rc]{\(1\,594\,323\)}}
\put(-4.8,-8){\makebox(0,0)[lc]{\(3^{13}\)}}
\put(-5.2,-9){\makebox(0,0)[rc]{\(4\,782\,969\)}}
\put(-4.8,-9){\makebox(0,0)[lc]{\(3^{14}\)}}
\put(-5.2,-10){\makebox(0,0)[rc]{\(14\,348\,907\)}}
\put(-4.8,-10){\makebox(0,0)[lc]{\(3^{15}\)}}
\put(-5.2,-11){\makebox(0,0)[rc]{\(43\,046\,721\)}}
\put(-4.8,-11){\makebox(0,0)[lc]{\(3^{16}\)}}
\put(-5.2,-12){\makebox(0,0)[rc]{\(129\,140\,163\)}}
\put(-4.8,-12){\makebox(0,0)[lc]{\(3^{17}\)}}
\put(-5.2,-13){\makebox(0,0)[rc]{\(387\,420\,489\)}}
\put(-4.8,-13){\makebox(0,0)[lc]{\(3^{18}\)}}
\put(-5.2,-14){\makebox(0,0)[rc]{\(1\,162\,261\,467\)}}
\put(-4.8,-14){\makebox(0,0)[lc]{\(3^{19}\)}}
\put(-5.2,-15){\makebox(0,0)[rc]{\(3\,486\,784\,401\)}}
\put(-4.8,-15){\makebox(0,0)[lc]{\(3^{20}\)}}
\put(-5,-15){\vector(0,-1){1}}

\put(-2,0.2){\makebox(0,0)[cb]{Root \(R=X_3^2\)}}
\put(-1.7,-1){\makebox(0,0)[lc]{Fork \(F=X_4^2\)}}
\put(-2,-1){\circle{0.3}}
\put(-2,-1){\line(-1,-1){1}}
\put(-3.3,-3){\makebox(0,0)[rc]{GS}}
\put(-3.3,-3.3){\makebox(0,0)[rt]{\(M_0^i\)}}
\put(-3,-2){\circle*{0.1}}
\put(-3,-2){\line(0,-1){1}}
\put(-3,-3){\circle{0.2}}

\multiput(-2,0)(0,-1){11}{\circle*{0.1}}

\put(-2,-1){\line(0,-1){2}}
\put(-1.8,-3){\makebox(0,0)[lc]{\(X_6^2\)}}
\put(-2,-3){\line(-1,-1){1}}
\put(-3.3,-5){\makebox(0,0)[rc]{ES \(1\)}}
\put(-3.3,-5.3){\makebox(0,0)[rt]{\(M_1^i\)}}
\put(-3,-4){\circle*{0.1}}
\put(-3,-4){\line(0,-1){1}}
\put(-3,-5){\circle{0.2}}

\put(-2,-3){\line(0,-1){2}}
\put(-1.8,-5){\makebox(0,0)[lc]{\(X_8^2\)}}
\put(-2,-5){\line(-1,-1){1}}
\put(-3.3,-7){\makebox(0,0)[rc]{ES \(2\)}}
\put(-3.3,-7.3){\makebox(0,0)[rt]{\(M_2^i\)}}
\put(-3,-6){\circle*{0.1}}
\put(-3,-6){\line(0,-1){1}}
\put(-3,-7){\circle{0.2}}

\put(-2,-5){\line(0,-1){2}}
\put(-1.8,-7){\makebox(0,0)[lc]{\(X_{10}^2\)}}
\put(-2,-7){\line(-1,-1){1}}
\put(-3.3,-9){\makebox(0,0)[rc]{ES \(3\)}}
\put(-3.3,-9.3){\makebox(0,0)[rt]{\(M_3^i\)}}
\put(-3,-8){\circle*{0.1}}
\put(-3,-8){\line(0,-1){1}}
\put(-3,-9){\circle{0.2}}

\put(-2,-7){\line(0,-1){2}}
\put(-1.8,-9){\makebox(0,0)[lc]{\(X_{12}^2\)}}
\put(-2,-9){\line(-1,-1){1}}
\put(-3.3,-11){\makebox(0,0)[rc]{ES \(4\)}}
\put(-3.3,-11.3){\makebox(0,0)[rt]{\(M_4^i\)}}
\put(-3,-10){\circle*{0.1}}
\put(-3,-10){\line(0,-1){1}}
\put(-3,-11){\circle{0.2}}

\put(-2,-9){\vector(0,-1){3}}
\put(-2,-12.3){\makebox(0,0)[cc]{\({}_\ast\mathcal{T}^2(R)\)}}
\put(-2,-12.7){\makebox(0,0)[cc]{Infinite}}
\put(-2,-13.1){\makebox(0,0)[cc]{mainline}}

\put(-2,0){\line(0,-1){1}}
\put(-2,-1){\line(1,-1){2}}
\put(0.2,-3){\makebox(0,0)[lc]{\(X_5^3\)}}
\put(0.2,-4){\makebox(0,0)[lc]{\(X_6^3\)}}
\put(-2,-1){\line(1,-2){1}}
\multiput(-1.05,-3.05)(0,-3){4}{\framebox(0.1,0.1){}}
\multiput(-1.05,-4.05)(0,-3){4}{\framebox(0.1,0.1){}}
\put(-1,-3){\line(0,-1){1}}
\put(-1,-4){\line(1,-2){1}}
\put(-0.1,-6.1){\framebox(0.2,0.2){}}
\put(0,-6.4){\makebox(0,0)[ct]{\(S_{0,0}^i\)}}
\put(-1,-4){\line(0,-1){2}}
\put(-1,-6){\line(0,-1){1}}
\put(-1,-7){\line(1,-2){1}}
\put(-0.1,-9.1){\framebox(0.2,0.2){}}
\put(0,-9.4){\makebox(0,0)[ct]{\(S_{0,1}^i\)}}
\put(-1,-7){\line(0,-1){2}}
\put(-1,-9){\line(0,-1){1}}
\put(-1,-10){\line(1,-2){1}}
\put(-0.1,-12.1){\framebox(0.2,0.2){}}
\put(0,-12.4){\makebox(0,0)[ct]{\(S_{0,2}^i\)}}
\put(-1,-10){\line(0,-1){2}}
\put(-1,-12){\line(0,-1){1}}
\put(-1,-13){\line(1,-2){1}}
\put(-0.1,-15.1){\framebox(0.2,0.2){}}
\put(0,-15.4){\makebox(0,0)[ct]{\(S_{0,3}^i\)}}
\put(-1,-13){\vector(0,-1){2}}
\put(-1,-15.1){\makebox(0,0)[ct]{GS}}

\put(0,-3){\line(0,-1){1}}
\put(0,-4){\line(1,-1){2}}
\put(2.2,-6){\makebox(0,0)[lc]{\(X_7^4\)}}
\put(2.2,-7){\makebox(0,0)[lc]{\(X_8^4\)}}
\put(0,-4){\line(1,-2){1}}
\multiput(0.95,-6.05)(0,-3){3}{\framebox(0.1,0.1){}}
\multiput(0.95,-7.05)(0,-3){3}{\framebox(0.1,0.1){}}
\put(1,-6){\line(0,-1){1}}
\put(1,-7){\line(1,-2){1}}
\put(1.9,-9.1){\framebox(0.2,0.2){}}
\put(2,-9.4){\makebox(0,0)[ct]{\(S_{1,0}^i\)}}
\put(1,-7){\line(0,-1){2}}
\put(1,-9){\line(0,-1){1}}
\put(1,-10){\line(1,-2){1}}
\put(1.9,-12.1){\framebox(0.2,0.2){}}
\put(2,-12.4){\makebox(0,0)[ct]{\(S_{1,1}^i\)}}
\put(1,-10){\line(0,-1){2}}
\put(1,-12){\line(0,-1){1}}
\put(1,-13){\line(1,-2){1}}
\put(1.9,-15.1){\framebox(0.2,0.2){}}
\put(2,-15.4){\makebox(0,0)[ct]{\(S_{1,2}^i\)}}
\put(1,-13){\vector(0,-1){2}}
\put(1,-15.1){\makebox(0,0)[ct]{ES1}}

\put(2,-15.9){\makebox(0,0)[ct]{Infinite branch for each state}}

\put(2,-6){\line(0,-1){1}}
\put(2,-7){\line(1,-1){2}}
\put(4.2,-9){\makebox(0,0)[lc]{\(X_9^5\)}}
\put(4.2,-10){\makebox(0,0)[lc]{\(X_{10}^5\)}}
\put(2,-7){\line(1,-2){1}}
\multiput(2.95,-9.05)(0,-3){2}{\framebox(0.1,0.1){}}
\multiput(2.95,-10.05)(0,-3){2}{\framebox(0.1,0.1){}}
\put(3,-10){\line(0,-1){2}}
\put(3,-9){\line(0,-1){1}}
\put(3,-10){\line(1,-2){1}}
\put(3.9,-12.1){\framebox(0.2,0.2){}}
\put(4,-12.4){\makebox(0,0)[ct]{\(S_{2,0}^i\)}}
\put(3,-12){\line(0,-1){1}}
\put(3,-13){\line(1,-2){1}}
\put(3.9,-15.1){\framebox(0.2,0.2){}}
\put(4,-15.4){\makebox(0,0)[ct]{\(S_{2,1}^i\)}}
\put(3,-13){\vector(0,-1){2}}
\put(3,-15.1){\makebox(0,0)[ct]{ES2}}

\put(4,-9){\line(0,-1){1}}
\put(4,-10){\line(1,-1){2}}
\put(6.2,-12){\makebox(0,0)[lc]{\(X_{11}^6\)}}
\put(6.2,-13){\makebox(0,0)[lc]{\(X_{12}^6\)}}
\put(4,-10){\line(1,-2){1}}
\put(4.95,-12.05){\framebox(0.1,0.1){}}
\put(4.95,-13.05){\framebox(0.1,0.1){}}
\put(5,-12){\line(0,-1){1}}
\put(5,-13){\line(1,-2){1}}
\put(5.9,-15.1){\framebox(0.2,0.2){}}
\put(6,-15.4){\makebox(0,0)[ct]{\(S_{3,0}^i\)}}
\put(5,-13){\vector(0,-1){2}}
\put(5,-15.1){\makebox(0,0)[ct]{ES3}}

\put(6,-12){\line(0,-1){1}}
\put(6,-13){\vector(1,-1){2}}
\put(9,-15.5){\makebox(0,0)[cc]{Infinite}}
\put(9,-15.9){\makebox(0,0)[cc]{maintrunk}}
\put(6,-13){\line(1,-2){1}}
\put(6.95,-15.05){\framebox(0.1,0.1){}}

\put(0,-6){\vector(-1,1){2.9}}
\put(2,-8.9){\vector(-4,3){5.0}}
\put(3.9,-12.1){\vector(-4,3){6.8}}
\put(5.8,-15.2){\vector(-4,3){8.4}}

\end{picture}

\end{figure}

}


{\tiny


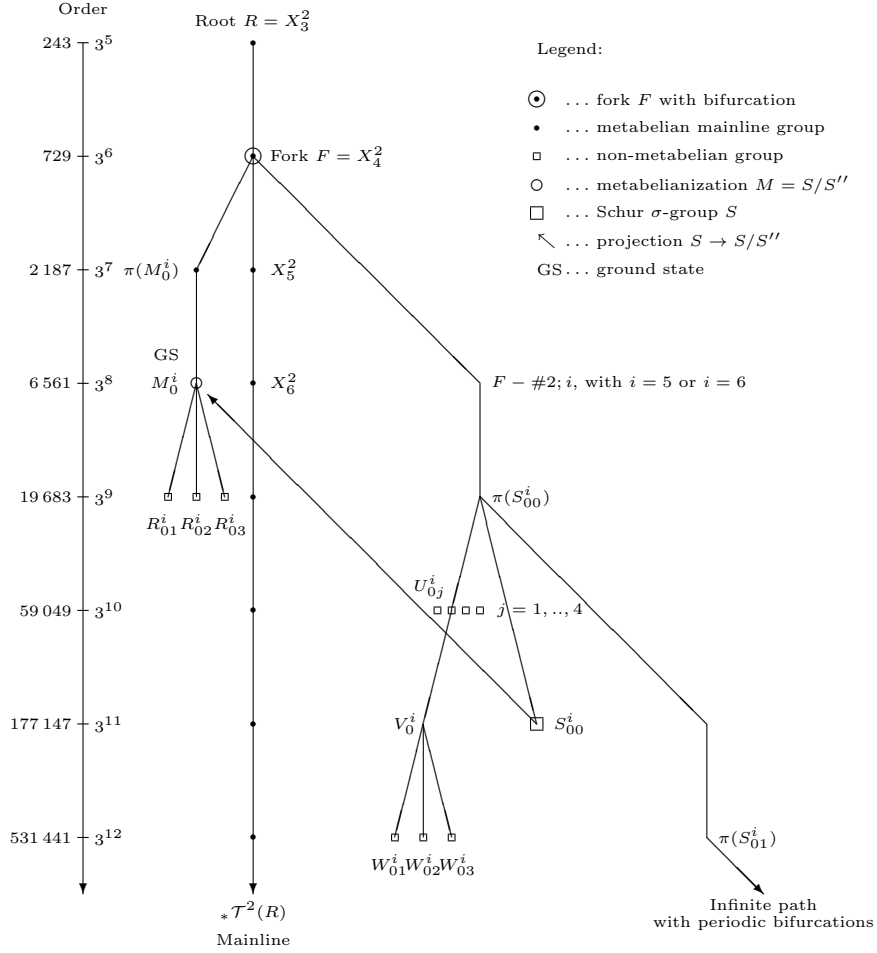
\begin{figure}[ht]
\caption{Topologies of \(3\)-tower groups with \textbf{complex} type H.4 in the GS on \({}_\ast\mathcal{T}(R)\)}
\label{fig:ComplexTopologies}

\setlength{\unitlength}{0.75cm}
\begin{picture}(18,17)(-8,-16)


\put(3,0){\makebox(0,0)[lt]{Legend:}}
\put(3,-1){\circle{0.3}}
\put(3,-1){\circle*{0.1}}
\put(3.5,-1){\makebox(0,0)[lc]{\(\ldots\) fork \(F\) with bifurcation}}
\put(3,-1.5){\circle*{0.1}}
\put(3.5,-1.5){\makebox(0,0)[lc]{\(\ldots\) metabelian mainline group}}
\put(2.95,-2.05){\framebox(0.1,0.1){}}
\put(3.5,-2){\makebox(0,0)[lc]{\(\ldots\) non-metabelian group}}
\put(3,-2.5){\circle{0.2}}
\put(3.5,-2.5){\makebox(0,0)[lc]{\(\ldots\) metabelianization \(M=S/S^{\prime\prime}\)}}
\put(2.9,-3.1){\framebox(0.2,0.2){}}
\put(3.5,-3){\makebox(0,0)[lc]{\(\ldots\) Schur \(\sigma\)-group \(S\)}}
\put(3,-3.5){\makebox(0,0)[lc]{\(\nwarrow\)}}
\put(3.5,-3.5){\makebox(0,0)[lc]{\(\ldots\) projection \(S\to S/S^{\prime\prime}\)}}
\put(3,-4){\makebox(0,0)[lc]{GS}}
\put(3.5,-4){\makebox(0,0)[lc]{\(\ldots\) ground state}}


\put(-5,0.5){\makebox(0,0)[cb]{Order}}
\put(-5,0){\line(0,-1){14}}
\multiput(-5.1,0)(0,-2){8}{\line(1,0){0.2}}
\put(-5.2,0){\makebox(0,0)[rc]{\(243\)}}
\put(-4.8,0){\makebox(0,0)[lc]{\(3^5\)}}
\put(-5.2,-2){\makebox(0,0)[rc]{\(729\)}}
\put(-4.8,-2){\makebox(0,0)[lc]{\(3^6\)}}
\put(-5.2,-4){\makebox(0,0)[rc]{\(2\,187\)}}
\put(-4.8,-4){\makebox(0,0)[lc]{\(3^7\)}}
\put(-5.2,-6){\makebox(0,0)[rc]{\(6\,561\)}}
\put(-4.8,-6){\makebox(0,0)[lc]{\(3^8\)}}
\put(-5.2,-8){\makebox(0,0)[rc]{\(19\,683\)}}
\put(-4.8,-8){\makebox(0,0)[lc]{\(3^9\)}}
\put(-5.2,-10){\makebox(0,0)[rc]{\(59\,049\)}}
\put(-4.8,-10){\makebox(0,0)[lc]{\(3^{10}\)}}
\put(-5.2,-12){\makebox(0,0)[rc]{\(177\,147\)}}
\put(-4.8,-12){\makebox(0,0)[lc]{\(3^{11}\)}}
\put(-5.2,-14){\makebox(0,0)[rc]{\(531\,441\)}}
\put(-4.8,-14){\makebox(0,0)[lc]{\(3^{12}\)}}
\put(-5,-14){\vector(0,-1){1}}

\put(-2,0){\line(0,-1){14}}
\multiput(-2,0)(0,-2){8}{\circle*{0.1}}

\put(-2,0.2){\makebox(0,0)[cb]{Root \(R=X_3^2\)}}
\put(-1.7,-2){\makebox(0,0)[lc]{Fork \(F=X_4^2\)}}
\put(-1.7,-4){\makebox(0,0)[lc]{\(X_5^2\)}}
\put(-1.7,-6){\makebox(0,0)[lc]{\(X_6^2\)}}
\put(-2,-2){\circle{0.3}}
\put(-2,-2){\line(-1,-2){1}}
\put(-3.3,-4){\makebox(0,0)[rc]{\(\pi(M_0^i)\)}}
\put(-3,-4){\circle*{0.1}}
\put(-3,-4){\line(0,-1){2}}
\put(-3.3,-5.5){\makebox(0,0)[rc]{GS}}
\put(-3.3,-6){\makebox(0,0)[rc]{\(M_0^i\)}}
\put(-3,-6){\circle{0.2}}
\put(-3,-6){\line(-1,-4){0.5}}
\put(-3,-6){\line(0,-1){2}}
\put(-3,-6){\line(1,-4){0.5}}
\multiput(-3.55,-8.05)(0.5,0){3}{\framebox(0.1,0.1){}}
\put(-3.6,-8.3){\makebox(0,0)[ct]{\(R_{01}^i\)}}
\put(-3,-8.3){\makebox(0,0)[ct]{\(R_{02}^i\)}}
\put(-2.4,-8.3){\makebox(0,0)[ct]{\(R_{03}^i\)}}

\put(-2,-14){\vector(0,-1){1}}
\put(-2,-15.3){\makebox(0,0)[cc]{\({}_\ast\mathcal{T}^2(R)\)}}
\put(-2,-15.8){\makebox(0,0)[cc]{Mainline}}

\put(2.2,-6){\makebox(0,0)[lc]{\(F-\#2;i\), with \(i=5\) or \(i=6\) }}
\put(-2,-2){\line(1,-1){4}}
\put(2,-6){\line(0,-1){2}}
\put(2,-8){\line(1,-4){1}}
\put(3.3,-12){\makebox(0,0)[lc]{\(S_{00}^i\)}}
\put(2.9,-12.1){\framebox(0.2,0.2){}}

\put(2,-8){\line(-1,-4){0.5}}
\put(1.4,-9.8){\makebox(0,0)[rb]{\(U_{0j}^i\)}}
\multiput(1.2,-10.05)(0.25,0){4}{\framebox(0.1,0.1){}}
\put(2.3,-10){\makebox(0,0)[lc]{\(j=1,..,4\)}}
\put(1.5,-10){\line(-1,-4){0.5}}

\put(0.9,-12){\makebox(0,0)[rc]{\(V_{0}^i\)}}
\put(1,-12){\line(-1,-4){0.5}}
\put(1,-12){\line(0,-1){2}}
\put(1,-12){\line(1,-4){0.5}}
\multiput(0.45,-14.05)(0.5,0){3}{\framebox(0.1,0.1){}}
\put(0.4,-14.3){\makebox(0,0)[ct]{\(W_{01}^i\)}}
\put(1,-14.3){\makebox(0,0)[ct]{\(W_{02}^i\)}}
\put(1.6,-14.3){\makebox(0,0)[ct]{\(W_{03}^i\)}}

\put(3,-12){\vector(-1,1){5.8}}

\put(2.2,-8){\makebox(0,0)[lc]{\(\pi(S_{00}^i)\)}}
\put(2,-8){\line(1,-1){4}}
\put(6,-12){\line(0,-1){2}}
\put(6.2,-14){\makebox(0,0)[lc]{\(\pi(S_{01}^i)\)}}
\put(6,-14){\vector(1,-1){1}}
\put(7,-15.2){\makebox(0,0)[cc]{Infinite path}}
\put(7,-15.5){\makebox(0,0)[cc]{with periodic bifurcations}}

\end{picture}

\end{figure}

}


\subsection{Details for the complex ground state}
\label{ss:ComplexTreeDetails}


\noindent
Remarks concerning Figure
\ref{fig:ComplexTopologies}:
A detailed tree diagram for the ground state (GS) of \textit{complex types} is drawn in Figure
\ref{fig:ComplexTopologies}.
It is pruned from simple types, and restricted to skeleton types and complex types.
Four different scenarios can be realized by the same diagram:
the two trees with fork \(F\) equal to \(Q=\langle 729,49\rangle\) or \(U=\langle 729,54\rangle\),
and each of them with two possible selections of the metabelianization \(M_0^i\),
sharing the same complex type in the GS.

\noindent
\(\bullet\)
\textbf{The \(Q\)-tree:}
Skeleton type is c.18, \((0122)\),
which forms the infinite mainline of the tree \(\mathcal{T}^2(R)\) with fixed coclass \(2\),
starting at the root \(R=\langle 243,6\rangle\),
passing the fork \(F=Q=\langle 729,49\rangle\) with nuclear rank \(\nu=2\) and bifurcation,
on the one hand with permanent step size \(s=1\) to
\(X_5^2=\langle 2187,285\rangle\), \(X_6^2=\langle 6561,2024\rangle\), etc.
\cite[Fig. 3, p. 151]{Ma2017},
and on the other hand with alternating step sizes \(s=2\) and \(s=1\)
on an infinite path with complex type H.4, \((2122)\), and with periodic bifurcations,
giving rise to unboundedly increasing coclass.
We are interested in two siblings of \(\langle 2187,285\rangle\) rather than in the mainline.
They lead to two arithmetically indistinguishable candidates for the metabelianization
\(M=\mathrm{Gal}(\mathrm{F}_3^2(K)/K)\) of the \(3\)-class field tower group.

\begin{enumerate}
\item
Either via the sibling \(\langle 2187,286\rangle\),
which is forbidden as \(M\) due to the lack of a proper \(\sigma\)-automorphism,
to \(M=M_0^5=\langle 6561,2030\rangle\).
\item
Or via the sibling \(\langle 2187,287\rangle\)
which is forbidden as \(M\) due to the lack of a proper \(\sigma\)-automorphism,
to \(M=M_0^6=\langle 6561,2035\rangle\).
\end{enumerate}

\noindent
In both cases, several configurations of \(3\)-class field tower groups \(S\) are possible.
For real quadratic base fields \(K\),
three terminal immediate non-metabelian descendants
\(R_{0j}^i=M_0^i-\#1;j\), \(j=1,2,3\),
provide the option of a Schur\(+1\) \(\sigma\)-group \(S\) with child topology.

\noindent
Similarly, we are interested in two siblings,
\(F-\#2;5=\langle 6561,614\rangle\) and \(F-\#2;6=\langle 6561,615\rangle\),
of \(F-\#2;1=\langle 6561,613\rangle\)
rather than in the skeleton type itself.


\renewcommand{\arraystretch}{1.1}

\begin{table}[ht]
\caption{Second ATI of \(R_{0j}^i\) for the \textbf{Ground State} of type H.4}
\label{tbl:SchurPlusOneH4}
\begin{center}
\begin{tabular}{|r|l||l|l|l|}
\hline
    &                    &                & \multicolumn{2}{|c|}{Stabilization}                     \\
 lo & id                 & Polarization   & Singular               & Regular                        \\
\hline
  8 & \(2030\vert 2035\) & \(32,(311)^3\) & \(111;(11)^9,(111)^3\) & \(\lbrack 21;(21)^3\rbrack^2\) \\
\hline
  9 & \(M_0^i-\#1;1\vert 2\) & \(32,(411)^3\) & \(111;(11)^9,(111)^3\) & \(\lbrack 21;(21)^3\rbrack^2\) \\
  9 & \(M_0^i-\#1;3\)        & \(32,(311)^3\) & \(111;(11)^9,(111)^3\) & \(\lbrack 21;(21)^3\rbrack^2\) \\
\hline
\end{tabular}
\end{center}
\end{table}


\noindent
\(\bullet\)
\textbf{The \(U\)-tree:}
Skeleton type is c.21, \((0231)\),
which forms the infinite mainline of the tree \(\mathcal{T}^2(R)\) with fixed coclass \(2\),
starting at the root \(R=\langle 243,8\rangle\),
passing the fork \(F=Q=\langle 729,54\rangle\) with nuclear rank \(\nu=2\) and  bifurcation,
on the one hand with permanent step size \(s=1\) to
\(X_5^2=\langle 2187,303\rangle\), \(X_6^2=\langle 6561,2050\rangle\), etc.
\cite[Fig. 4, p. 152]{Ma2017},
and on the other hand with alternating step sizes \(s=2\) and \(s=1\)
on an infinite path with complex type G.16, \((4231)\), and with periodic bifurcations,
causing unboundedly increasing coclass.
We are interested in two siblings of \(\langle 2187,303\rangle\) rather than in the mainline.
They lead to two arithmetically indistinguishable candidates for the metabelianization
\(M=\mathrm{Gal}(\mathrm{F}_3^2(K)/K)\) of the \(3\)-class field tower group.

\begin{enumerate}
\item
Either via the sibling \(\langle 2187,301\rangle\),
which is forbidden as \(M\) due to the lack of a proper \(\sigma\)-automorphism,
to \(M=M_0^5=\langle 6561,2048\rangle\).
\item
Or via the sibling \(\langle 2187,305\rangle\)
which is forbidden as \(M\) due to the lack of a proper \(\sigma\)-automorphism,
to \(M=M_0^6=\langle 6561,2058\rangle\).
\end{enumerate}

\noindent
In both cases, several configurations of \(3\)-class field tower groups \(S\) are possible.
For real quadratic base fields \(K\),
two terminal immediate non-metabelian descendants
\(R_{0j}^i=M_0^i-\#1;j\), \(j=1,2\),
provide the option of a Schur\(+1\) \(\sigma\)-group \(S\) with child topology.

\noindent
Similarly, we are interested in two siblings,
\(F-\#2;5=\langle 6561,619\rangle\) and \(F-\#2;6=\langle 6561,623\rangle\),
of \(F-\#2;1=\langle 6561,621\rangle\)
rather than in the skeleton type itself.


\renewcommand{\arraystretch}{1.1}

\begin{table}[ht]
\caption{Second ATI of \(R_{0j}^i\) for the \textbf{Ground State} of type G.16}
\label{tbl:SchurPlusOneG16}
\begin{center}
\begin{tabular}{|r|l||l|l|}
\hline
 lo & id                 & Polarization   & Stabilization                  \\
\hline
  8 & \(2048\vert 2058\) & \(32,(311)^3\) & \(\lbrack 21;(21)^3\rbrack^3\) \\
\hline
  9 & \(M_0^i-\#1;1\)        & \(32,(411)^3\) & \(\lbrack 21;(21)^3\rbrack^3\) \\
  9 & \(M_0^i-\#1;2\)        & \(32,(311)^3\) & \(\lbrack 21;(21)^3\rbrack^3\) \\
\hline
\end{tabular}
\end{center}
\end{table}



\section{Computational techniques}
\label{s:Computations}

\subsection{The fundamental database of Bush}
\label{ss:Bush}

\noindent
On 11 July 2015, M. R. Bush kindly
shared an extensive database with us, in private communication.
This information was the computational background for the article
\cite{BBH2021},
and extended the range of our own tables in
\cite{Ma2012a, Ma2014}
by a factor of \(100\).
It was created in several months of CPU time on a cluster of parallel supercomputers,
and so it would have been impossible for us to reconstruct this numerical data.
The associated \texttt{README\_real.txt} file explains the contents of the database:

\bigskip

\lq\lq The Magma file \texttt{ipad\_freq\_real.m} contains two lists called \texttt{disclist} and \texttt{ipadlist}.
Each list contains \(185\) entries. 

Each entry in \texttt{ipadlist} is the IPAD, for \(p = 3\), of a \textit{real quadratic field} \(K\)
with \(3\)-class group of rank \(2\) and discriminant \(d_K < 10^9\).
The IPAD of such a field \(K\) is a \(5\)-tuple in which
the first entry is the \(3\)-class group of \(K\) and
the remaining four entries are the \(3\)-class groups of the four unramified cyclic cubic extensions of \(K\).

The \(i\)-th entry in \texttt{disclist} is the complete list of discriminants \(d_K\) of real quadratic fields
with \(3\)-class group of rank \(2\), \(d_K < 10^9\)
and whose IPAD appears as the \(i\)-th entry in \texttt{ipadlist}.

All class groups have been computed using Magma \texttt{V2.19-5}.

Note:
Taken together,
the discriminants appearing in \texttt{disclist} form a complete list of discriminants 
for real quadratic fields \(K\) with \(3\)-class group of rank \(2\) and \(d_K < 10^9\).
There are \(481\,756\) such fields.

Last updated: Sat, 11 July 2015.\rq\rq


\renewcommand{\arraystretch}{1.1}

{\small

\begin{table}[ht]
\caption{List of IPADs with first component \([ 3, 3 ]\) for \(d_K<10^9\)}
\label{tbl:IPAD}
\begin{center}
\begin{tabular}{|r|r|r|r|l|c|}
\hline
Rel & Abs & NumDisc & MinDisc            & IPAD (last four components)                                   & State \\
\hline
  1 &    1 & 208236 &        \(32\,009\) &  ( [ 3, 3 ], [ 3, 3 ], [ 3, 3 ], [ 3, 9 ] )                         & \\
  2 &    2 & 122955 &       \(142\,097\) &  ( [ 3, 3 ], [ 3, 3 ], [ 3, 3 ], [ 3, 3, 3 ] )                      & \\
  3 &    4 &  26678 &        \(62\,501\) &  ( [ 3, 3 ], [ 3, 3 ], [ 3, 3 ], [ 9, 9 ] )                         & \\
  4 &    5 &  13712 &       \(422\,573\) &  ( [ 3, 3, 3 ], [ 3, 9 ], [ 3, 9 ], [ 3, 9 ] )                      & \\
  5 &    6 &  11780 &       \(494\,236\) &  ( [ 3, 3 ], [ 3, 3 ], [ 3, 3 ], [ 9, 27 ] )                        & \\
  6 &    9 &   6691 &       \(631\,769\) &  ( [ 3, 3, 3 ], [ 3, 3, 3 ], [ 3, 9 ], [ 3, 9 ] )                   & \\
  7 &   10 &   6583 &       \(957\,013\) &  ( [ 3, 3, 3 ], [ 3, 3, 3 ], [ 3, 3, 3 ], [ 3, 9 ] )                & \\
  8 &   11 &   4377 &       \(540\,365\) &  ( [ 3, 9 ], [ 3, 9 ], [ 3, 9 ], [ 9, 9 ] )                         & \\
  9 &   12 &   4318 &       \(534\,824\) &  ( [ 3, 3, 3 ], [ 3, 9 ], [ 3, 9 ], [ 9, 9 ] )                      & \\
 10 & \textbf{16} & 1958 &      \(342\,664\) &  \(\mathbf{( [ 3, 9 ], [ 3, 9 ], [ 3, 9 ], [ 9, 27 ] )}\)       & GS \\
 11 & \textbf{17} & 1880 &   \(1\,162\,949\) &  \(\mathbf{( [ 3, 3, 3 ], [ 3, 9 ], [ 3, 9 ], [ 9, 27 ] )}\)    & GS \\
 12 &   18 &   1636 &       \(214\,712\) &  ( [ 3, 9 ], [ 3, 9 ], [ 3, 9 ], [ 3, 9 ] )                         & \\
 13 &   19 &   1410 &       \(710\,652\) &  ( [ 3, 3, 3 ], [ 3, 3, 3 ], [ 9, 9 ], [ 9, 9 ] )                   & \\
 14 &   21 &   1251 &    \(1\,535\,117\) &  ( [ 3, 3, 3 ], [ 3, 3, 3 ], [ 9, 9 ], [ 9, 27 ] )                  & \\
 15 &   25 &    921 &    \(2\,905\,160\) &  ( [ 3, 3 ], [ 3, 3 ], [ 3, 3 ], [ 27, 27 ] )                       & \\
 16 &   32 &    391 &   \(10\,200\,108\) &  ( [ 3, 3 ], [ 3, 3 ], [ 3, 3 ], [ 27, 81 ] )                       & \\
 17 &   38 &    234 &    \(8\,321\,505\) &  ( [ 3, 3, 3 ], [ 3, 3, 3 ], [ 9, 27 ], [ 9, 27 ] )                 & \\
 18 &   43 &    146 &    \(1\,001\,957\) &  ( [ 3, 9 ], [ 3, 9 ], [ 3, 9 ], [ 27, 27 ] )                       & \\
 19 &   45 &    138 &   \(13\,714\,789\) &  ( [ 3, 3, 3 ], [ 3, 9 ], [ 3, 9 ], [ 27, 27 ] )                    & \\
 20 &   48 &    101 &   \(17\,802\,872\) &  ( [ 3, 3, 3 ], [ 3, 3, 3 ], [ 9, 9 ], [ 27, 27 ] )                 & \\
 21 & \textbf{55} &   81 &  \(26\,889\,637\) &  \(\mathbf{( [ 3, 9 ], [ 3, 9 ], [ 3, 9 ], [ 27, 81 ] )}\)      & ES1 \\
 22 & \textbf{59} &   66 &  \(70\,539\,596\) &  \(\mathbf{( [ 3, 3, 3 ], [ 3, 9 ], [ 3, 9 ], [ 27, 81 ] )}\)   & ES1 \\
 23 &   66 &     40 &    \(8\,491\,713\) &  ( [ 3, 3, 3 ], [ 3, 3, 3 ], [ 9, 27 ], [ 27, 27 ] )                & \\
 24 &   70 &     31 &   \(27\,970\,737\) &  ( [ 3, 3, 3 ], [ 3, 3, 3 ], [ 9, 9 ], [ 27, 81 ] )                 & \\
 25 &   74 &     25 &   \(40\,980\,808\) &  ( [ 3, 3 ], [ 3, 3 ], [ 3, 3 ], [ 81, 81 ] )                       & \\
 26 &   76 &     23 &    \(8\,127\,208\) &  ( [ 3, 3, 3 ], [ 3, 3, 3 ], [ 9, 27 ], [ 27, 81 ] )                & \\
 27 &   86 &     12 &   \(37\,304\,664\) &  ( [ 3, 3 ], [ 3, 3 ], [ 3, 3 ], [ 81, 243 ] )                      & \\
 28 &  114 &      5 &  \(174\,458\,681\) &  ( [ 3, 3, 3 ], [ 3, 9 ], [ 3, 9 ], [ 81, 81 ] )                    & \\
 29 &  115 &      5 &  \(116\,043\,324\) &  ( [ 3, 9 ], [ 3, 9 ], [ 3, 9 ], [ 81, 81 ] )                       & \\
 30 &  116 &      4 &  \(131\,279\,821\) &  ( [ 3, 3, 3 ], [ 3, 3, 3 ], [ 9, 9 ], [ 81, 243 ] )                & \\
 31 &  129 &      3 &  \(343\,438\,961\) &  ( [ 3, 3, 3 ], [ 3, 3, 3 ], [ 9, 9 ], [ 81, 81 ] )                 & \\
 32 &\textbf{130} &    3 & \(124\,813\,084\) &  \(\mathbf{( [ 3, 9 ], [ 3, 9 ], [ 3, 9 ], [ 81, 243 ] )}\)     & ES2 \\
 33 &  151 &      2 &  \(180\,527\,768\) &  ( [ 3, 3, 3 ], [ 3, 3, 3 ], [ 27, 27 ], [ 27, 27 ] )               & \\
 34 &\textbf{165} &    1 & \(336\,698\,284\) &  \(\mathbf{( [ 3, 3, 3 ], [ 3, 9 ], [ 3, 9 ], [ 81, 243 ] )}\)  & ES2 \\
 35 &\textbf{170} &    1 & \(705\,576\,037\) &  \(\mathbf{( [ 3, 3, 3 ], [ 3, 9 ], [ 3, 9 ], [ 243, 729 ] )}\) & ES3 \\
\hline
\end{tabular}
\end{center}
\end{table}

}

\bigskip
\noindent
Using a Magma program script \texttt{SiftRealIPADs.m},
we retrieved the following statistics of IPADs in Table
\ref{tbl:IPAD},
for the upper bound \(d_K < 10^9\),
restricted to the \textit{elementary bicyclic} first IPAD component \([ 3, 3 ]\):
\lq\lq Rel\rq\rq\ denotes the relative counter of the \(35\) sifted IPADs,
as opposed to the absolute counter
\lq\lq Abs\rq\rq\
among all \(185\) IPADs.
\lq\lq NumDisc\rq\rq\
is the number of discriminants with given IPAD.
\lq\lq MinDisc\rq\rq\
is the minimal discriminant with assigned IPAD.
Since the first component \([ 3, 3 ]\) is fixed,
only the last four components of the IPAD are given.
\textbf{Boldface} IPADs are crucial. \\
Number of all IPADs: \(185\),
Number of sifted IPADs: \(35\), \\
Total number of discriminants: \(415\,698\).

\bigskip
\noindent
In order to identify the relevant IPADs
of vertices on the Q-tree and U-tree,
corresponding to the six transfer kernel types
under investigation (Formula
\eqref{eqn:SimpleTKTComplexTKT}),
we define:
Let \(\varepsilon\) be the \textit{number of
elementary tricyclic components} \([ 3, 3, 3 ]\) of an IPAD.
By the \lq\lq main theorem on class and coclass from IPAD\rq\rq\ 
and its corollary in
\cite[\S\ 4, Thm. 2 and Cor. 1]{AoMa2025},
it is well known that
\begin{itemize}
\item
an IPAD with (at least) three elementary bicyclic components \([ 3, 3 ]\) is due to a group of maximal class, \(\mathrm{cc}=1\)
(Abs = 1,2,4,6,25,32,74,86),
\item
an IPAD with \(\varepsilon=3\) is due to the sporadic group \(\langle 243,4\rangle\) and its descendants (Abs = 10),
\item
an IPAD with \(\varepsilon=2\) is either due to the sporadic Schur \(\sigma\)-group \(\langle 243,7\rangle\)
(Abs = 9)
or to the group \(\langle 243,3\rangle\) and its descendants,
in particular to all groups with \(\mathrm{cc}\ge 3\)
(Abs = 19,21,38,48,66,70,76,116,129,151),
\item
among the remaining IPADs with \(\varepsilon=1\), 
Abs = 5 is due to the sporadic Schur \(\sigma\)-group \(\langle 243,5\rangle\),
the polarization of Abs = 12,45,114 is \textbf{homo}cyclic,
and only \textbf{Abs = 17,59,165,170} belong to vertices with \textbf{hetero}cyclic polarization on the Q-tree,
\item
among the remaining IPADs with \(\varepsilon=0\), 
Abs = 18 is due to the sporadic group \(\langle 243,9\rangle\) and its descendants,
the polarization of Abs = 11,43,115 is \textbf{homo}cyclic,
and only \textbf{Abs = 16,55,130} belong to vertices with \textbf{hetero}cyclic polarization on the U-tree.
\end{itemize}


\subsection{The principalization type}
\label{ss:Capitulation}

\noindent
By means of \texttt{GetDiscsForFixedIPAD.m}, a Magma program script,
we retrieved the discriminants in ascending order for the IPADs
\[\textbf{Abs = 16,17,55,59,130,165,170}.\]
The TKT was computed with \texttt{RQGroundDisc.m}, the ATI2 and Length with \texttt{SecondATIExtended.m}.

\noindent
In Table
\ref{tbl:IPAD16},
the initial 15 discriminants with IPAD Abs = 16
\(( [ 3, 9 ], [ 3, 9 ], [ 3, 9 ], [ 9, 27 ] )\)
are classified
according to their transfer kernel type
TKT E.8 or E.9 or G.16, in the ground state GS, \(n = 0\).
The corresponding metabelianizations \(M\) are vertices on the U-tree.


\renewcommand{\arraystretch}{1.1}

{\small

\begin{table}[ht]
\caption{List of discriminants associated with IPAD Abs = 16}
\label{tbl:IPAD16}
\begin{center}
\begin{tabular}{|r|r|l|l|c|}
\hline
 No &              Disc & Factors                          & TKT  & Length \\
\hline
  1 &      \(\mathbf{342\,664}\) & \(2^3\cdot 7\cdot 29\cdot 211\) & E.9  & \textbf{3} \\
  2 &   \(1\,452\,185\) & \(5\cdot 7\cdot 41491\)                  & E.9  & 3 \\
  3 &   \(1\,787\,945\) & \(5\cdot 353\cdot 1013\)                 & E.9  & 3 \\
  4 &   \(\mathbf{4\,760\,877}\) & \(3\cdot 11\cdot 89\cdot 1621\) & E.9  & \textbf{2} \\
  5 &   \(4\,861\,720\) & \(2^3\cdot 5\cdot 19\cdot 6397\)         & E.9  & 3 \\
  6 &   \(5\,976\,988\) & \(2^2\cdot 1494247\)                     & E.9  & 3 \\
  7 &   \(\mathbf{6\,098\,360}\) & \(2^3\cdot 5\cdot 152459\)      & E.8  & \textbf{3} \\
  8 &   \(6\,652\,929\) & \(3\cdot 2217643\)                       & E.9  & 2 \\
  9 &   \(7\,100\,889\) & \(3\cdot 19\cdot 124577\)                & E.8  & 3 \\
 10 &   \(7\,358\,937\) & \(3\cdot 71\cdot 34549\)                 & E.9  & 2 \\
 11 &   \(8\,079\,101\) & prime                                    & E.9  & 3 \\
 12 &   \(\mathbf{8\,632\,716}\) & \(2^2\cdot 3\cdot 719393\)      & E.8  & \textbf{2} \\
 13 &   \(\mathbf{8\,711\,453}\) & \(947\cdot 9199\)               & G.16 & \textbf{3} \\
 14 &   \(9\,129\,480\) & \(2^3\cdot 3\cdot 5\cdot 76079\)         & E.9  & 2 \\
 15 &   \(\mathbf{9\,448\,265}\) & \(5\cdot 1889653\)              & G.16 & \textbf{2 or 3} \\
\hline
\end{tabular}
\end{center}
\end{table}

}


\noindent
In Table
\ref{tbl:IPAD17},
the initial 10 discriminants with IPAD Abs = 17 
\(( [ 3, 3, 3 ], [ 3, 9 ], [ 3, 9 ], [ 9, 27 ] )\)
are classified
according to their transfer kernel type
TKT E.6 or E.14 or H.4, in the ground state GS, \(n = 0\).
The corresponding metabelianizations \(M\) are vertices on the Q-tree.

\renewcommand{\arraystretch}{1.1}

{\small

\begin{table}[ht]
\caption{List of discriminants associated with IPAD Abs = 17}
\label{tbl:IPAD17}
\begin{center}
\begin{tabular}{|r|r|l|l|c|}
\hline
 No &              Disc & Factors                                 & TKT  & Length \\
\hline
  1 &   \(\mathbf{1\,162\,949}\) & \(23\cdot 59\cdot 857\)        & H.4  & \textbf{2 or 3} \\
  2 &   \(\mathbf{2\,747\,001}\) & \(3\cdot 19\cdot 48193\)       & H.4  & \textbf{3} \\
  3 &   \(3\,122\,232\) & \(2^3\cdot 3\cdot 19\cdot 41\cdot 167\) & H.4  & 2 or 3 \\
  4 &   \(\mathbf{3\,918\,837}\) & \(3\cdot 13\cdot 100483\)      & E.14 & \textbf{2} \\
  5 &   \(4\,074\,493\) & \(19\cdot 131\cdot 1637\)               & H.4  & 2 or 3 \\
  6 &   \(\mathbf{5\,264\,069}\) & \(139\cdot 37871\)             & E.6  & \textbf{3} \\
  7 &   \(6\,946\,573\) & \(29\cdot 31\cdot 7727\)                & E.6  & 3 \\
  8 &   \(\mathbf{7\,153\,097}\) & \(7\cdot 613\cdot 1667\)       & E.6  & \textbf{2} \\
  9 &   \(8\,897\,192\) & \(2^3\cdot 163\cdot 6823\)              & E.14 & 2 \\
 10 &   \(\mathbf{9\,433\,849}\) & \(2549\cdot 3701\)             & E.14 & \textbf{3} \\
\hline
\end{tabular}
\end{center}
\end{table}

} 
 
\noindent
In Table
\ref{tbl:IPAD55},
the initial 18 discriminants with IPAD Abs = 55
\(( [ 3, 9 ], [ 3, 9 ], [ 3, 9 ], [ 27, 81 ] )\)
are classified
according to their transfer kernel type
TKT E.8 or E.9 or G.16, in the first excited state ES 1, \(n = 1\).
The corresponding metabelianizations \(M\) are vertices on the U-tree.


\renewcommand{\arraystretch}{1.1}

{\small

\begin{table}[ht]
\caption{List of discriminants associated with IPAD Abs = 55}
\label{tbl:IPAD55}
\begin{center}
\begin{tabular}{|r|r|l|l|c|}
\hline
 No &              Disc & Factors                             & TKT  & Length \\
\hline

  1 &  \(\mathbf{26\,889\,637}\) & prime                               & E.8  & \textbf{2} \\
  2 &  \(\mathbf{59\,479\,964}\) & \(2^2\cdot 347\cdot 42853\)         & G.16 & \textbf{2 or 3} \\
  3 &  \(\mathbf{79\,043\,324}\) & \(2^2\cdot 19760831\)               & E.9  & \textbf{2} \\
  4 &  \(98\,755\,469\) & \(29\cdot 139\cdot 24499\)                   & E.8  & 2 \\
  5 & \(111\,121\,161\) & \(3\cdot 1163\cdot 31849\)                   & E.9  & 2 \\
  6 & \(135\,445\,241\) & prime                                        & E.9  & 2 \\
  7 & \(147\,910\,989\) & \(3\cdot 79\cdot 624097\)                    & E.8  & 2 \\
  8 & \(155\,191\,657\) & \(17\cdot 83\cdot 109987\)                   & E.9  & 2 \\
  9 & \(157\,423\,029\) & \(3\cdot 83\cdot 632221\)                    & G.16 & 2 or 3 \\
 10 & \(\mathbf{178\,243\,036}\) & \(2^2\cdot 113\cdot 139\cdot 2837\) & E.9  & \textbf{3} \\
 11 & \(188\,823\,317\) & \(8293\cdot 22769\)                          & E.8  & 2 \\
 12 & \(\mathbf{209\,483\,033}\) & prime                               & G.16 & \(\mathbf{\ge 3}\) \\
 13 & \(227\,396\,348\) & \(2^2\cdot 56849087\)                        & G.16 & 2 or 3 \\
 14 & \(230\,668\,493\) & \(11\cdot19\cdot 619\cdot 1783\)             & E.9  & \\
 15 & \(248\,917\,036\) & \(2^2\cdot 887\cdot 70157\)                  & E.9  & \\
 16 & \(249\,304\,648\) & \(2^3\cdot 29\cdot 613\cdot 1753\)           & E.9  & \\
 17 & \(264\,062\,393\) & \(7\cdot 37723199\)                          & E.9  & \\
 18 & \(\mathbf{292\,399\,937}\) & \(37\cdot 7902701\)                 & E.8  & \textbf{3} \\
\hline
\end{tabular}
\end{center}
\end{table}

}


\noindent
In Table
\ref{tbl:IPAD59},
the initial 7 discriminants with IPAD Abs = 59 
\(( [ 3, 3, 3 ], [ 3, 9 ], [ 3, 9 ], [ 27, 81 ] )\)
are classified
according to their transfer kernel type
TKT E.6 or E.14 or H.4, in the first excited state ES 1, \(n = 1\).
The corresponding metabelianizations \(M\) are vertices on the Q-tree.

\renewcommand{\arraystretch}{1.1}

{\small

\begin{table}[ht]
\caption{List of discriminants associated with IPAD Abs = 59}
\label{tbl:IPAD59}
\begin{center}
\begin{tabular}{|r|r|l|l|c|}
\hline
 No &              Disc & Factors                             & TKT  & Length \\
\hline
  1 &  \(\mathbf{70\,539\,596}\) & \(2^2\cdot 17\cdot 1037347\)        & E.14 & \textbf{3} \\
  2 &  \(\mathbf{75\,393\,861}\) & \(3\cdot 17\cdot 941\cdot 1571\)    & E.6  & \textbf{3} \\
  3 & \(\mathbf{111\,046\,577}\) & \(181\cdot 199\cdot 3083\)          & E.6  & \textbf{2} \\
  4 & \(\mathbf{113\,284\,396}\) & \(2^2\cdot 17\cdot 347\cdot 4801\)  & E.14 & \textbf{2} \\
  5 & \(\mathbf{126\,691\,957}\) & \(7\cdot 103\cdot 199\cdot 883\)    & H.4  & \textbf{3} \\
  6 & \(136\,970\,636\) & \(2^2\cdot 11\cdot 103\cdot 30223\)          & E.14 & 2 \\
  7 & \(\mathbf{170\,356\,565}\) & \(5\cdot 19\cdot 1793227\)          & H.4  & \textbf{2 or 3} \\
\hline
\end{tabular}
\end{center}
\end{table}

}


\section{Real quadratic prototypes}
\label{s:ProtoTypes}

\noindent
For each assigned transfer kernel type 
\(\varkappa(k)\in\lbrace (1122), (2122), (3122), (1231), (2231), (4231)\rbrace\)
and each foregiven excited state of (logarithmic) abelian type invariants (ATI), or transfer target type (TTT),
\(\alpha(k)\in\lbrace \lbrack(n+3,n+2),111,21,21\rbrack; \lbrack(n+3,n+2),21,21,21\rbrace\)
with non-negative integers \(n\ge 0\),
the \textbf{prototype} is the minimal positive quadratic fundamental discriminant \(d\)
such that the real quadratic field \(k=\mathbb{Q}(\sqrt{d})\)
has TKT \(\varkappa(k)\) and TTT \(\alpha(k)\).

\bigskip
\noindent
The \textbf{Ground State} is characterized by
the \textit{polarization}
\((32)\) in the ATI of the first order
\begin{equation}
\label{eqn:Ground}
\alpha(k)=\lbrack 32,111,21,21\rbrack, \text{ respectively } \lbrack 32,21,21,21\rbrack,
\text{ see Table }\ref{tbl:Ground}\text{ under the GRH},
\end{equation}
whereas the \textit{stabilization}
\((111,21,21)\), respectively \((21,21,21)\), remains the same for all states.


\renewcommand{\arraystretch}{1.1}

\begin{table}[ht]
\caption{Prototypes of the \textbf{Ground State} computed using
\cite{Fi2001,MAGMA2026,Sh1964,Ma2015c}}

\label{tbl:Ground}
\begin{center}
\begin{tabular}{|rr|l|c|c|}
\hline
 No. &           \(d\) & TKT  & \(\ell_3(k)\) & Reference \\
\hline
\hline
   6 & \(5\,264\,069\) & E.6  &         \(3\) & Tbl. \ref{tbl:IPAD17} \\
   8 & \(7\,153\,097\) & E.6  &         \(2\) & Tbl. \ref{tbl:IPAD17} \\
\hline
   4 & \(3\,918\,837\) & E.14 &         \(2\) & Tbl. \ref{tbl:IPAD17} \\
  10 & \(9\,433\,849\) & E.14 &         \(3\) & Tbl. \ref{tbl:IPAD17} \\
\hline 
   1 & \(1\,162\,949\) & H.4  & \(2\) or \(3\)& Tbl. \ref{tbl:IPAD17} \\
   2 & \(2\,747\,001\) & H.4  &         \(3\) & Tbl. \ref{tbl:IPAD17} \\
  25 &\(23\,064\,965\) & H.4  &       \(\ge 3\)& Extension \\
  30 &\(30\,118\,269\) & H.4  & Schur \(\ge 3\)& Extension \\
\hline  
\hline 
   7 & \(6\,098\,360\) & E.8  &         \(3\) & Tbl. \ref{tbl:IPAD16} \\
  12 & \(8\,632\,716\) & E.8  &         \(2\) & Tbl. \ref{tbl:IPAD16} \\
\hline 
   1 &    \(342\,664\) & E.9  &         \(3\) & Tbl. \ref{tbl:IPAD16} \\
   4 & \(4\,760\,877\) & E.9  &         \(2\) & Tbl. \ref{tbl:IPAD16} \\
\hline 
  13 & \(8\,711\,453\) & G.16 &         \(3\) & Tbl. \ref{tbl:IPAD16} \\
  15 & \(9\,448\,265\) & G.16 & \(2\) or \(3\)& Tbl. \ref{tbl:IPAD16} \\
\hline
\end{tabular}
\end{center}
\end{table}


\bigskip
\noindent
The \textbf{First Excited State} is characterized by
the \textit{polarization}
\((43)\) in the ATI of the first order

\begin{equation}
\label{eqn:Excited}
\alpha(k)=\lbrack 43,111,21,21\rbrack \text{ respectively } \lbrack 43,21,21,21\rbrack,
\text{ see Table }\ref{tbl:Excited}\text{ under the GRH}.
\end{equation}


\renewcommand{\arraystretch}{1.1}

\begin{table}[ht]
\caption{Prototypes of the \textbf{First Excited State} computed by B. Allombert}
\label{tbl:Excited}
\begin{center}
\begin{tabular}{|rr|l|c|c|}
\hline
 No. &             \(d\) & TKT  & \(\ell_3(k)\) & Reference \\
\hline
\hline
   2 &  \(75\,393\,861\) & E.6  &         \(3\) & Tbl. \ref{tbl:IPAD59} \\
   3 & \(111\,046\,577\) & E.6  &         \(2\) & Tbl. \ref{tbl:IPAD59} \\
\hline
   1 &  \(70\,539\,596\) & E.14 &         \(3\) & Tbl. \ref{tbl:IPAD59} \\
   4 & \(113\,284\,396\) & E.14 &         \(2\) & Tbl. \ref{tbl:IPAD59} \\
\hline 
   5 & \(126\,691\,957\) & H.4  &         \(3\) & Tbl. \ref{tbl:IPAD59} \\
   7 & \(170\,356\,565\) & H.4  & \(2\) or \(3\)& Tbl. \ref{tbl:IPAD59} \\
  14 & \(216\,353\,320\) & H.4  &     \(\ge 3\) & Extension \\  
\hline 
\hline  
   1 &  \(26\,889\,637\) & E.8  &         \(2\) & Tbl. \ref{tbl:IPAD55} \\
  18 & \(292\,399\,937\) & E.8  &         \(3\) & Tbl. \ref{tbl:IPAD55} \\
\hline 
   3 &  \(79\,043\,324\) & E.9  &         \(2\) & Tbl. \ref{tbl:IPAD55} \\
  10 & \(178\,243\,036\) & E.9  &         \(3\) & Tbl. \ref{tbl:IPAD55} \\
\hline 
   2 &  \(59\,479\,964\) & G.16 & \(2\) or \(3\)& Tbl. \ref{tbl:IPAD55} \\
  12 & \(209\,483\,033\) & G.16 &      \(\ge 3\)& Tbl. \ref{tbl:IPAD55} \\
\hline
\end{tabular}
\end{center}
\end{table}


\noindent
The root region up to the logarithmic order \(\mathrm{lo}=20\)
of two infinite descendant trees
\cite{Nm1975,Nm1989,Ob1990}
is drawn in Figure
\ref{fig:TreeStrucSimple}
for the pruned tree \({}^\ast\mathcal{T}(R)\),
which is restricted to the \textit{skeleton} type and \textit{simple} types,
and in Figure
\ref{fig:TreeStrucComplex}
for the pruned tree \({}_\ast\mathcal{T}(R)\),
which is restricted to the \textit{skeleton} type and the \textit{complex} type.
For the types E.6, E.14 and H.4, the root is \(R=\langle 243,6\rangle\) with type c.18,
and for the types E.8, E.9 and G.16, the root is \(R=\langle 243,8\rangle\) with type c.21.


\section{Conclusion}
\label{s:Conclusion}

\noindent
The investigation of non-metabelian \(3\)-class field towers with length \(\ell_3(k)=3\)
started in 2012 in cooperation with M. R. Bush
\cite[Thm. 4.1 and Cor. 4.1.1, pp. 774--775]{BuMa2015},
where the erroneous claim \(\ell_3(k)=2\)
for the \textbf{imaginary} quadratic field \(k=\mathbb{Q}(\sqrt{-9748})\)
with TKT E.9 in the GS \(n=0\) in
\cite[p. 41]{SoTa1934}
was disproved by the rigorous verification that
either \(S=S_0^3\) or \(S=S_0^4\), rather than \(S=M\),
presented at the West Coast Number Theory Conference (WCNT) in Asilomar, December 2013.

In 2015, 
\cite[pp. 184--193]{Ma2015a}
we extended the proof for \(\ell_3(k)=3\) of 2012
to all four simple TKTs, E.6, E.8, E.9, E.14, of imaginary quadratic fields,
up to ES \(n\le 7\), i.e. \(\mathrm{cl}\le 19\), \(\mathrm{cc}\le 10\), \(\mathrm{lo}\le 29\),
introducing the cover \(\mathrm{cov}(M)\),
presented at the 29th Journ\'ees Arithm\'etiques (JA) in Debrecen, July 2015.
We also drew a more detailed version of the present Figure
\ref{fig:TreeStrucSimple}
(bounded by \(\mathrm{lo}=20\))
in the Figures 8--9 on pp. 188--189
(bounded by \(\mathrm{lo}=14\)),
and we studied the normal lattice of the Schur \(\sigma\)-groups \(S_0^i\) and \(S_1^i\)
in the Figures 10--11 on pp. 191--192.

Arithmetical prototypes up to ES \(n\le 4\)
for all imaginary quadratic fields \(k=\mathbb{Q}(\sqrt{d})\), \(-10^8<d<0\),
were computed for the Figures 3--4 on pp. 754--755 in
\cite{Ma2015b},
presented at the 1st International Conference on Goups and Algebras (ICGA) in Shanghai, July 2015.
See also the Figures 1--2 on pp. 24--25 in
\cite{Ma2015c}.

It should be pointed out that the present Theorem
\ref{thm:SimpleStageCriterion}
on the length \(\ell_3(k)\) of \textbf{real} quadratic fields with \textbf{simple} TKT
was proved on 5 Oct 2025 for all excited states (ES \(n\)) with \(n\ge 1\),
whereas it was known for the ground state (GS) with \(n=0\) already in 
\cite[Thm. 6.3, pp. 298--299]{Ma2015d}
for the \(U\)-tree alone, in
\cite[Thm. 7.5--7.12, pp. 159--166]{Ma2017}
for both, the \(Q\)-tree and the \(U\)-tree,
and the trees with minimal positive discriminants \(d>0\) of prototypes were drawn in
\cite[Fig. 3--4, pp. 151--152]{Ma2017}.

Due to the highly confusing branches of \textbf{complex} TKTs, Theorem
\ref{thm:ComplexStageCriterion}
was proved on 15 Oct 2025 for bounded state parameter \(0\le n\le 4\) only.
With second order ATIs, only implications can be proven rather than equivalences.
\textbf{TODO:} For an extension to \(n\ge 5\),
modified infinite limit groups would be required,
maybe \(3\)-adic Lie groups like \(\mathrm{SL}_2(\mathbb{Z}_3)\).

It was possible to prove necessary and sufficient criteria
for all excited states (ES) parametrized with state parameter \(n\ge 0\).
Figure
\ref{fig:TreeStrucSimple}
ostensively reveals that,
in the case of non-coincidence, \(S\ne M\),
the connecting path between \(M\) and \(S\)
consists of an ascending part with exclusive step size \(s=1\)
from the metabelianization \(M=S/S^{\prime\prime}\) to the \textbf{fork} \(F=R(-\#1;1)\),
and a descending part with alternating step sizes \(s=2\) and \(s=1\),
the so-called \textbf{structure elements} (SE),
from the fork \(F\) to the Schur \(\sigma\)-group \(S\).
The ground state (GS) gives rise to logarithmic order
\(\mathrm{lo}(S)=8\) without SE, and
the \(n\)-th excited state (ES \(n\))
requires exactly \(n\) SE
and leads to logarithmic order \(\mathrm{lo}(S)=8+3n\).

Examples for the distinction between the two possible tower lengths 
\(\ell_3(k)\in\lbrace 2,3\rbrace\)
for real quadratic base fields \(k\)
are given for the ground state (GS) in
\cite[Thm. 7.8, pp. 162--163, and Thm. 7.12, pp. 165--166]{Ma2017}.
For the first excited state (ES \(1\)), we have for instance
three stages for \(d=70\,539\,596\) of type E.14,
and for \(d=75\,393\,861\) of type E.6,
but only two stages for \(d=26\,889\,637\) of type E.8.
For the second excited state (ES \(2\)), we have
only two stages for \(d=336\,698\,284\) of type E.14
(ES 2 is not contained in our tables).

The tree structure in Figure
\ref{fig:TreeStrucComplex}
is more complicated, due to \textit{infinite branches}.
In contrast to the \textit{simple types} in section E of Scholz and Taussky,
where a real quadratic field \(k=\mathbb{Q}(\sqrt{d})\) with
tower length \(\ell_3(k)=3\) has a Schur or Schur\(+1\) \(\sigma\)-group
\(S=\mathrm{Gal}(\mathrm{F}_3^\infty(k)/k)\)
as the Galois group of its maximal unramified pro-\(3\)-extension
and the path between \(S\) and \(M=S/S^{\prime\prime}\) is a \textit{fork topology}
as described above, 
the \textit{complex types} in the sections G and H of Scholz and Taussky
admit a variety of different scenarios,
most frequently a \textit{child topology} with an immediate (Schur+1)-descendant \(S=M-\#1\) of \(M\)
and thus necessarily precise length \(\ell_3(k)=3\),
for instance 
\(d=2\,747\,001\) and \(d=126\,691\,957\) of type H.4, and
\(d=8\,711\,453\) of type G.16,
or less frequently a \textit{fork topology} with either a (Schur+1) \(\sigma\)-group \(S\),
such as \(d=23\,064\,965\) and \(d=216\,353\,320\) of type H.4, and
\(d=209\,483\,033\) of type G.16,
or quite sparsely with a Schur \(\sigma\)-group \(S\),
e.g. \(d=30\,118\,269\) of type H.4.
All the latter are located on infinite branches with presumably unbounded \(\ell_3(k)\ge 3\).


\renewcommand{\arraystretch}{1.1}

\begin{table}[ht]
\caption{Soluble length for \textbf{sporadic} and \textbf{periodic} type H.4}
\label{tbl:SolubleLengthH4}
\begin{center}
\begin{tabular}{|r|r|r||c|r||c|r|c|r|c|r|c|r|}
\hline
    &    &    & \multicolumn{2}{|c||}{}         & \multicolumn{8}{|c|}{Periodic}                                                                               \\
    &    &    & \multicolumn{2}{|c||}{Sporadic} & \multicolumn{2}{|c|}{GS} & \multicolumn{2}{|c|}{ES1} & \multicolumn{2}{|c|}{ES2} & \multicolumn{2}{|c|}{ES3} \\
 lo & cl & cc & sl &                      \(a\) & sl &               \(a\) & sl &                \(a\) & sl &                \(a\) & sl &                \(a\) \\
\hline
  8 &  5 &  3 & 3 &                          9 &   &                      &   &                       &   &                       &   &                       \\
 11 &  7 &  4 & 3 &                         12 & 3 &                   13 &   &                       &   &                       &   &                       \\
 14 &  9 &  5 & 3 &                         15 & 3 &                   16 & 3 &                    17 &   &                       &   &                       \\
\hline
 17 & 11 &  6 & 4 &                         18 & 3 &                   19 & 3 &                    20 & 3 &                    21 &   &                       \\
 20 & 13 &  7 & 4 &                         21 & 3 &                   22 & 3 &                    23 & 3 &                    24 & 3 &                    25 \\
 23 & 15 &  8 & 4 &                         24 & 4 &                   25 & 3 &                    26 & 3 &                    27 & 3 &                    28 \\
 26 & 17 &  9 & 4 &                         27 & 4 &                   28 & 3 &                    29 & 3 &                    30 & 3 &                    31 \\
 29 & 19 & 10 & 4 &                         30 & 4 &                   31 & 4 &                    32 & 3 &                    33 & 3 &                    34 \\
\hline
 32 & 21 & 11 & 5 &                         33 & 4 &                   34 & 4 &                    35 & 3 &                    36 & 3 &                    37 \\
 35 & 23 & 12 & 5 &                         36 & 4 &                   37 & 4 &                    38 & 4 &                    39 & 3 &                    40 \\
 38 & 25 & 13 & 5 &                         39 & 4 &                   40 & 4 &                    41 & 4 &                    42 & 3 &                    43 \\
 41 & 27 & 14 & 5 &                         42 & 4 &                   43 & 4 &                    44 & 4 &                    45 & 4 &                    46 \\
 44 & 29 & 15 & 5 &                         45 & 5 &                   46 & 4 &                    47 & 4 &                    48 & 4 &                    49 \\
 47 & 31 & 16 & 5 &                         48 & 5 &                   49 & 4 &                    50 & 4 &                    51 & 4 &                    52 \\
 50 & 33 & 17 & 5 &                         51 & 5 &                   52 & 4 &                    53 & 4 &                    54 & 4 &                    55 \\
 53 & 35 & 18 & 5 &                         54 & 5 &                   55 & 4 &                    56 & 4 &                    57 & 4 &                    58 \\
 56 & 37 & 19 & 5 &                         57 & 5 &                   58 & 5 &                    59 & 4 &                    60 & 4 &                    61 \\
 59 & 39 & 20 & 5 &                         60 & 5 &                   61 & 5 &                    62 & 4 &                    63 & 4 &                    64 \\
 62 & 41 & 21 & 5 &                         63 & 5 &                   64 & 5 &                    65 & 4 &                    66 & 4 &                    67 \\
\hline
 65 & 43 & 22 & 6 &                         66 & 5 &                   67 & 5 &                    68 & 4 &                    69 & 4 &                    70 \\
\hline
\end{tabular}
\end{center}
\end{table}

\noindent
In
\cite{BaBu2007},
a \textit{sporadic variant} of the TKT H.4, \(\varkappa=(4111)\sim (2122)\),
with \textit{stable} ATI \(\alpha=\lbrack (1^3)^3,21\rbrack\)
is analyzed with the aid of the \textit{special linear group}
of dimension \(2\) over the \(3\)-adic integers, \(\mathrm{SL}_2(\mathbb{Z}_3)\).
It is the \(3\)-principalization type of all members of the
infinite descendant tree of the root \(\langle 243,4\rangle\),
whose unique immediate \(\sigma\)-descendant is Ascione's group \(N=\langle 729,45\rangle\)
\cite{AHL1977}.
The pruned subtree \(\mathcal{T}_\sigma(N)\) of all \(\sigma\)-descendants of \(N\)
contains an infinite sequence \((G_m)_{m\ge 2}\) of non-metabelian Schur \(\sigma\)-groups
possessing an unbounded soluble length \(\mathrm{sl}(G_m)=\lfloor\log_2(3m+3)\rfloor\)
\cite[Thm. 2.1, p. 160]{BaBu2007}.

The location of these groups in the tree is given by the parametrized relative identifier
\begin{equation}
\label{eqn:BaBu}
N(-\#2;1-\#1;1)^{m-2}-\#2;2, \qquad m\ge 2.
\end{equation}

\noindent
\textbf{TODO:} It seems plausible that other \(3\)-adic Lie groups can be used to investigate
the Schur \(\sigma\)-groups \(S_{n,t}^i\) with the
\textit{periodic variant} of the TKT H.4, \(\varkappa=(2122)\),
and with \textit{variable} ATI
depending on the order \(n\ge 0\) of the excited state (ES \(n\)):
\(\alpha=\lbrack (n+3,n+2),1^3,(21)^2\rbrack\)
with \textbf{hetero}cyclic polarization \((n+3,n+2)\).

In Table
\ref{tbl:SolubleLengthH4},
we compare the logarithmic order lo, nilpotency class cl, coclass cc, soluble length sl,
and the \(3\)-valuation \(a:=v_3(A_m)\) of the order of the automorphism group \(A_m:=\mathrm{ord}(G_m)\)
with the corresponding values of the ground state (GS) and the first three excited states (ES)
of the Schur \(\sigma\)-groups \(S_{n,t}^i\).
\textbf{TODO:} Deterministic laws for their soluble length seem plausible.


\section{Acknowledgements}
\label{s:Thanks}

\noindent
We are indebted to Michael Raymond Bush at the
Washington and Lee University (WLU) in Lexington, Virginia, USA,
for making available an extensive database of quadratic fundamental discriminants
with assigned abelian type invariants of the first order.
Sincere thanks are given to Mike F. Newman  at the
Australian National University (ANU) in Canberra, Capital Territory, Australia, for the infinite limit groups
which are crucial in the proofs of the main theorems.
We thank Bill Allombert at the University of Bordeaux, Aquitaine, France,
for computational aid in determining decisive abelian type invariants of the second order
with PARI/GP for Table
\ref{tbl:Excited}.

The authors are deeply indebted to the anonymous referee
for his expertise and helpful suggestions concerning the entire exposition.


\section{Data availability}
\label{s:Data}

\noindent
Magma program scripts and output
may be requested from the second author via email.


\section{Conflict of interests}
\label{s:Conflict}

\noindent
The authors declare that no conflict of interests arises from this article.


\section{Funding Declaration}
\label{s:Support}

\noindent
The second author gratefully acknowledges financial support by
the Austrian Science Fund (FWF): projects J0497-PHY, P26008-N25,
and by the Research Executive Agency of the European Union (EUREA):
project Horizon Europe 2021--2027.



\end{document}